\documentclass[english]{article}
\usepackage[latin9]{inputenc}
\usepackage{geometry}
\usepackage{amsmath}
\usepackage{amsthm}
\usepackage{amssymb}
\usepackage{setspace}
\makeatletter
\theoremstyle{plain}
\newtheorem*{thm*}{\protect\theoremname}
\theoremstyle{plain}
\newtheorem{thm}{\protect\theoremname}
\theoremstyle{remark}
\newtheorem{claim}[thm]{\protect\claimname}

\date{}

\makeatother

\usepackage{babel}
\providecommand{\claimname}{Claim}
\providecommand{\theoremname}{Theorem}

\begin{document}
\title{A grid generalisation of the Kruskal-Katona theorem}
\author{Eero R\"{a}ty\thanks{Centre for Mathematical Sciences, Wilberforce Road, Cambridge CB3 0WB, UK, epjr2@cam.ac.uk}}
\maketitle
\begin{abstract}
For a set $A\subseteq\left[k\right]^{n}=\left\{ 0,\dots,k-1\right\} ^{n}$,
we define the $d$-shadow of $A$ to be the set of points obtained by flipping to zero one of the non-zero 
coordinates of some point in $A$. 
Let $\left[k\right]_{r}^{n}$ be the set of those points 
in $\left[k\right]^{n}$ with exactly $r$ non-zero coordinates. Given the size of $A$, 
how should we choose $A\subseteq\left[k\right]_{r}^{n}$
so as to minimise the $d$-shadow? Note that the case $k=2$ is answered
by the Kruskal-Katona theorem. 

Our aim in this paper is to give an exact answer to this question.
In particular, we show that the sets $\left[t\right]_{r}^{n}$ are
extremal for every $t$. We also give an exact answer to the 'unrestricted'
question when we just have $A\subseteq\left[k\right]^{n}$, showing
for example that the set of points with at least $r$ zeroes is extremal
for every $r$. 
\end{abstract}

\section{Introduction }

Let $\left\{ 0,1\right\} ^{n}$ be the set of all sequences $x=x_{1}\dots x_{n}$
of length $n$ with $x_{i}\in\left\{ 0,1\right\} $ for all $i$.
For $A\subseteq\left\{ 0,1\right\} ^{n}$, define the \textit{lower
shadow} of $A$ to be the set of points obtained from any of its point
by changing one coordinate of any point $x\in A$ from $1$ to $0$,
and denote it by $\partial^{-}A$. Define the \textit{rank} of a point
$x$ to be $w\left(x\right)=\left|\left\{ i:x_{i}=1\right\} \right|$,
and define $\left\{ 0,1\right\} _{r}^{n}=\left\{ x\in\left\{ 0,1\right\} ^{n}:w\left(x\right)=r\right\} $.
Note that the lower shadow operator decreases the rank of a point
by one. 

One can similarly define the \textit{upper shadow }of $A$ 
to be the set of points obtained from any of its point
by changing one coordinate of any point $x\in A$ from $0$ to $1$,
and denote the upper shadow by $\partial^{+}A$. Again,
it is clear that the upper shadow operator increases the rank of a
point by one. 

For a given $r$, it is natural to ask how should one choose $A\subseteq\left\{ 0,1\right\} _{r}^{n}$
of a given size in order to minimise the size of the lower shadow.
This question was answered by Kruskal \cite{key-51} and Katona \cite{key-6}.
Define the \textit{colexicographic order} $\leq_{c}$ on $\left\{ 0,1\right\} _{r}^{n}$
as follows. For points $x,y\in\left\{ 0,1\right\} _{r}^{n}$, let
$X=\left\{ i:x_{i}=1\right\} $ and $Y=\left\{ i:y_{i}=1\right\} $,
and  we set $x\leq_{c}y$ if $X=Y$ or $\max\left(X\Delta Y\right)\in Y$.
The Kruskal-Katona theorem states that for a set $A\subseteq\left\{ 0,1\right\} _{r}^{n}$
of a given size, the lower shadow of $A$ is minimised when $A$ is
chosen to be an initial segment of colexicographic order.

Define the \textit{lexicographic order} $\leq_{l}$ on $\left\{ 0,1\right\} _{r}^{n}$
by setting $x\leq_{l}y$ if $X=Y$ or $\min\left(X\Delta Y\right)\in X$,
where $X$ and $Y$ are defined as above. For $x\in\left\{ 0,1\right\} ^{n}$,
define $x^{c}\in\left\{ 0,1\right\} ^{n}$ to be the point obtained
by taking $\left(x^{c}\right)_{i}=1-x_{i}$ for all $i$. For a given
$A\subseteq\left\{ 0,1\right\} _{r}^{n}$, define $\overline{A}$
by setting $\overline{A}=\left\{ x^{c}:x\in A\right\} $. Note that
we have $\left|\overline{A}\right|=\left|A\right|$ and $\overline{A}\subseteq\left\{ 0,1\right\} _{n-r}^{n}$.
It is also easy to verify that $\partial^{+}A=\overline{\partial^{-}\overline{A}}$,
and hence it follows that $\left|\partial^{+}A\right|=\left|\partial^{-}\overline{A}\right|$.
Thus the question on minimising the size of the upper shadow in $\left\{ 0,1\right\} _{r}^{n}$
can be transformed into a question on minimising the size of the lower
shadow in $\left\{ 0,1\right\} _{n-r}^{n}$. Hence the Kruskal-Katona
theorem implies that the initial segments of the lexicographic order
minimise the upper shadow. 

There are many natural generalisations of the lower and upper shadow
for points in $\left[k\right]^{n}$, and one such generalisation can
be obtained in the following way. Let $\left[k\right]=\left\{ 0,\dots,k-1\right\} $
and $\left[k\right]^{n}=\left\{ 0,\dots,k-1\right\} ^{n}$. Define
the rank of a point $x\in\left[k\right]^{n}$ to be $w\left(x\right)=\left|\left\{ i:x_{i}\geq1\right\} \right|$.
For $0\leq r\leq n$ set $\left[k\right]_{r}^{n}=\left\{ x\in\left[k\right]^{n}:w\left(x\right)=r\right\} $. 

Define the \textit{$d$-shadow} of a point $x\in\left[k\right]^{n}$
to be the set of those points obtained from $x$ by flipping one of
the coordinates of $x$ which is in $\left\{ 1,\dots,k-1\right\} $
to $0$. We denote the $d$ shadow of a point $x$ by $d\left(\left\{ x\right\} \right)$.
For $A\subseteq\left[k\right]^{n}$ we define the $d$-shadow of $A$
by setting $d\left(A\right)=\bigcup_{x\in A}d\left(\left\{ x\right\} \right)$.
For example, $d\left(\left\{ 012\right\} \right)=\left\{ 002,010\right\} $
and $d\left(\left\{ 000\right\} \right)=\emptyset$. Note that the
rank of a point in $d\left(\left\{ x\right\} \right)$ is one lower
than the rank of $x$. It is clear that $d$ agrees with $\partial^{-}$
when $k=2$, so this operator indeed generalises the lower shadow. 

There is a superficial resemblance to a result of Clements, as we
now describe. Define the \textit{$d^{+}$-shadow }of a point $x\in\left[k\right]^{n}$
by setting $d^{+}\left(\left\{ x\right\} \right)$ to be the set of
those points obtained from $x$ by changing one of the coordinates
of $x$ which equals $0$ to any number in $\left\{ 1,\dots,k-1\right\} $,
and we set $d\left(A\right)=\bigcup_{x\in A}d\left(\left\{ x\right\} \right)$.
Again it is clear that the rank of a point  in $d^{+}\left(\left\{ x\right\} \right)$
is one larger than the rank of $x$. 

Clements \cite{key-7} found an order in $\left[k\right]_{r}^{n}$
whose initial segments have minimal $d^{+}$-shadow. Recall that the
ordinary lower and upper shadow can be related to each other by using
the fact that $\partial^{+}A=\overline{\partial^{-}\overline{A}}$.
However, for $k\geq3$, it is clear that there is no similar natural
relation between $d^{+}$- and $d$-shadows. That is, given the Clements'
result for $d^{+}$-shadow, there seems to be no way to deduce results
related to $d$-shadow. 

There is also a superficial resemblance to the Clements-Lindstr\"{o}m
Theorem \cite{key-1}. Let $k_{1},\dots,k_{n}$ be integers such that
$1\leq k_{1}\leq\dots\leq k_{n}$, and let $F$ be the set of all
integer sequences $\left(a_{1},\dots,a_{n}\right)$ with $0\leq a_{i}\leq k_{i}$
for all $i$. Define the shadow operator $\Gamma$ by setting 
\[
\Gamma\left(\left(a_{1},\dots a_{n}\right)\right)=\left\{ \left(a_{1}-1,a_{2}\dots,a_{n}\right),\left(a_{1},a_{2}-1,\dots,a_{n}\right)\dots,\left(a_{1},a_{2}\dots,a_{n}-1\right)\right\} \cap F,
\]
and $\Gamma\left(A\right)=\bigcup_{a\in A}\Gamma\left(a\right)$ for
$A\subseteq F$. Let $F_{r}$ be the set of those sequences $\left(a_{1},\dots,a_{n}\right)\in F$
with $\sum_{i=1}^{n}a_{i}=r$. Generalise the lexicographic order
by writing $\left(a_{1},\dots,a_{n}\right)<_{l}\left(b_{1},\dots,b_{n}\right)$
if there exists $i$ such that $a_{j}=b_{j}$ for all $j<i$ and $a_{i}<b_{i}$.
The Clements-Lindstr\"{o}m theorem states the initial segments of the
lexicographic order minimise the size of the $\Gamma$-shadow on $F_{r}$. 

The aim of this paper  is to find an order on $\left[k\right]_{r}^{n}$
whose initial segments have minimal $d$-shadow. In fact, we do this
by first solving the unrestricted version, i.e.\ we find an order on
$\left[k\right]^{n}$ whose initial segments have minimal $d$-shadow.
Once we have proved the result for $\left[k\right]^{n}$, the result
for $\left[k\right]_{r}^{n}$ follows easily. 

We start by defining the order, whose initial segments minimise the
$d$-shadow for the unrestricted version. For each $i$ define $R_{i}\left(x\right)=\left\{ j:x_{j}=i\right\} $.
For fixed $k$, define an order $\leq$ on $\left[k\right]^{n}$ by
setting $x\leq y$ if $x=y$ or one of the following conditions holds
\begin{enumerate}
\item $\left|R_{0}\left(x\right)\right|>\left|R_{0}\left(y\right)\right|$ 
\item $\left|R_{0}\left(x\right)\right|=\left|R_{0}\left(y\right)\right|$
and for the largest $i$ satisfying $R_{i}\left(x\right)\neq R_{i}\left(y\right)$
we have $\max\left(R_{i}\left(x\right)\Delta R_{i}\left(y\right)\right)\in R_{i}\left(y\right)$. 
\end{enumerate}
As usual, we say that $x<y$ if $x\leq y$ and $x\neq y$. Now we
are ready to state the unrestricted version of our theorem. 
\begin{thm}
Let $A$ be a subset of $\left[k\right]^{n}$, and let $B$ be an
initial segment of the $\leq$-order in $\left[k\right]^{n}$ with
$\left|B\right|=\left|A\right|$. Then $\left|d\left(A\right)\right|\geq\left|d\left(B\right)\right|$. 
\end{thm}
The proof of Theorem 1 is an inductive proof. As the proof of Theorem
1 is rather long, we split it into four subsections. In the first
subsection, we introduce certain codimension-1 compression operators,
and we prove that they cannot increase the size of the shadow. In
particular, because of the compression operators, it suffices to prove
Theorem 1 for those sets $A$ that are stable under the compression
operators. 

Compression operators have been much utilised in other papers as well,
see e.g.\ \cite{key-40,key-2}. For example, there are very straightforward
proofs using compression operators for Harper's vertex-isoperimetric
inequality on the hypercube and for the edge-isoperimetric inequality
on the hypercube. In these examples it is very straightforward to
prove the desired isoperimetric inequality for the sets that are stable
under the compression operators. However, proving the inequality for
the sets that are stable under the compression operators is highly
nontrivial in our main theorem. In particular, the main part of the
proof consists of dealing with the sets that are stable under the
compression operators and proving that Theorem 1 holds for such sets. 

Even though $n=1$ is the base case in the proof, it turns out to
be convenient to consider the case $n=2$ individually as well. Hence
we consider the case $n=2$ in the second subsection. This special
case turns out to be reasonably straightforward when restricted to
the sets that are stable under the compression operators. 

We start the proof of the general case $n\geq3$ by making some observations
on the structure of the sets that are stable under the compression
operators. As an example, we prove that it suffices to restrict our
attention to the sets $A$ for which there exists $r$ such that $B_{\geq r}\subseteq A\subseteq B_{\geq r+1}$
and which are also stable under the compression operators. Hence we
may write $A$ as $A=B_{\geq r}\cup D$ for some $D\subseteq B_{r+1}$,
and the fact that $A$ is stable under the compression operators also
restricts the structure of $D$. 

Let $A$ be a set of the form $A=B_{\geq r}\cup D$ for some $D\subseteq B_{r+1}$.
Note that we have $\left|d\left(A\right)\right|=\left|B_{\geq r+1}\right|+\left|d\left(D\right)\right|$,
and hence we can focus on analysing $d\left(D\right)$. The fourth
subsection is dedicated to analysing $d\left(D\right)$ by using previous
and some new structural observations. At this stage of the proof,
we need to split into smaller subcases depending on the size and the
structure of $A$. 

Denote the restriction of the $\leq$-order on $\mathbb{N}_{r}^{n}=\left\{ x=x_{1}\dots x_{n}:w\left(x\right)=r\right\} $
also by $\leq$, where we define $\mathbb{N}=\left\{ 0,1,\dots\right\} $.
This is well-defined, as it is easy to check that $\left[m\right]_{r}^{n}$
is an initial segment of the restriction of $\leq$ on $\left[k\right]_{r}^{n}$
for all $k\ge m$, and the orders coincide on $\left[m\right]_{r}^{n}$.
Now we can state our main theorem.
\begin{thm}
Let $A$ be a subset of $\left[k\right]_{r}^{n}$, and let $B$ be
an initial segment of the $\leq$-order in $\left[k\right]_{r}^{n}$
with $\left|B\right|=\left|A\right|$. Then $\left|d\left(A\right)\right|\geq\left|d\left(B\right)\right|$.
\end{thm}
We end this section by introducing some notation that we use throughout
the paper. We write $\left[n\right]^{\left(r\right)}=\left\{ A\subseteq\left[n\right]:\left|A\right|=r\right\} $
and $\left[n\right]^{\left(\leq r\right)}=\left\{ A\subseteq\left[n\right]:\left|A\right|\leq r\right\} $.
For convenience, we often write $B_{r}=\left[k\right]_{n-r}^{n}=\left\{ x\in\left[k\right]^{n}:\left|R_{0}\left(x\right)\right|=r\right\}$
and $B_{\geq r}=\bigcup_{i=r}^{n}B_{i}$ for the set of sequences
with exactly $r$ zeroes and at least $r$ zeroes respectively. Note
that $B_{r}$ depends on the ground set $\left[k\right]$, but since
the value of $k$ is often clear the dependence is not highlighted.
For $x\in\left[k\right]^{n}$, we write $d\left(x\right)$ instead
of $d\left(\left\{ x\right\} \right)$.

For $a_{i}\in\left[k\right]$ and for positive integers $t_{i}\in\mathbb{N}$
we define $\left(t_{1}\cdot a_{1}\right)\left(t_{2}\cdot a_{2}\right)\dots\left(t_{r}\cdot a_{r}\right)$
to represent the point $a_{1}\dots a_{1}a_{2}\dots a_{2}\dots a_{r}\dots a_{r}$,
which has $t_{1}$ $a_{1}$'s immediately followed by $t_{2}$ $a_{2}$'s,
and so on. For example, we have $\left(3\cdot0\right)\left(2\cdot4\right)56=0004456$.

For $y=y_{1}\dots y_{n-1}\in\left[k\right]^{n-1}$, $s\in\left[n\right]$
and $t\in\left[k\right]$ define 
\[
t_{s}y=y_{1}\dots y_{s-1}ty_{s}\dots y_{n-1},
\]
and for $Y\subseteq\left[k\right]^{n-1}$ we write $t_{s}Y=\left\{ t_{s}y\,:\,y\in Y\right\} $.
Note that we have $t_{s}x\leq t_{s}y$ if and only if $x\leq y$. 

Finally define the \textit{binary order }$\leq_{bin}$ on ${\cal P}\left(\mathbb{N}\right)$
by setting $X\leq_{b}Y$ if $X=Y$ or $\max\left(X\Delta Y\right)\in Y$.
We write $X<_{bin}Y$ if $X\leq_{bin}Y$ and $X\neq Y$. Note that
the second condition in the definition of the $\leq$-order is equivalent
to saying that for the largest $i$ for which $R_{i}\left(x\right)\neq R_{i}\left(y\right)$
we have $R_{i}\left(x\right)<_{bin}R_{i}\left(y\right)$. 

\section{Proof of Theorem 1}

For convenience, we say that a set $A\subseteq\left[k\right]^{n}$
is \textit{extremal }if the size of $d\left(A\right)$ is minimal
among all subsets of $\left[k\right]^{n}$ of the same size. The proof
of Theorem 1 is an inductive proof. Note that the case $n=1$ is trivial.
Throughout the Section 2 we assume that Theorem 1 holds for $\left[k\right]^{n-1}$,
and our aim is to prove it for $\left[k\right]^{n}$. 

\subsection{The compression operators}

For $A\subseteq\left[k\right]^{n}$, $t\in\left[k\right]$ and $s\in\left[n\right]$
define 
\[
A_{s,t}=\left\{ y\in\left[k\right]^{n-1}:t_{s}y\in A\right\} .
\]
Let $B_{s,t}\subseteq\left[k\right]^{n-1}$ be an initial segment
of the $\leq$-order with $\left|B_{s,t}\right|=\left|A_{s,t}\right|$,
and set $C_{s,t}=t_{s}B_{s,t}$. Define the \textit{$C_{s}$-compression
}of $A$ by setting $C_{s}\left(A\right)=\bigcup_{t=0}^{k-1}C_{s,t}$.
We start by proving that $C_{s}$-compressions cannot increase the
size of the $d$-shadow.
\addtocounter{thm}{-2}
\begin{claim}
For all $A\subseteq\left[k\right]^{n}$ and $s\in\left[n\right]$
we have $\left|C_{s}\left(A\right)\right|=\left|A\right|$ and $\left|d\left(A\right)\right|\geq\left|d\left(C_{s}\left(A\right)\right)\right|$.
\end{claim}

\begin{proof}
For a given $s$, note that the sets $\left(C_{s,t}:t\in\left[k\right]\right)$ are pairwise disjoint, 
as every $x\in C_{s,t}$ satisfies $x_{s}=t$.
Since $\left|C_{s,t}\right|=\left|A_{s,t}\right|$ for all $t\in\left[k\right]$,
it follows that $\left|C_{s}\left(A\right)\right|=\left|A\right|$. 

Note that we have

\begin{equation}
d\left(C_{s}\left(A\right)\right)=d\left(\bigcup_{t=0}^{k-1}t_{s}B_{s,t}\right)=\left(\left(\bigcup_{t=1}^{k-1}0_{s}B_{s,t}\right)\cup0_{s}d\left(B_{s,0}\right)\right)\cup\left(\bigcup_{t=1}^{k-1}t_{s}d\left(B_{s,t}\right)\right),\label{eq:1}
\end{equation}
and similarly we have 
\begin{equation}
d\left(A\right)=d\left(\bigcup_{t=0}^{k-1}t_{s}A_{s,t}\right)=\left(\left(\bigcup_{t=1}^{k-1}0_{s}A_{s,t}\right)\cup0_{s}d\left(A_{s,0}\right)\right)\cup\left(\bigcup_{t=1}^{k-1}t_{s}d\left(A_{s,t}\right)\right).\label{eq:2}
\end{equation}

Observe that the $k$ sets 
\[
\left(\bigcup_{t=1}^{k-1}0_{s}B_{s,t}\right)\cup0_{s}d\left(B_{s,0}\right),\,\,1_{s}d\left(B_{s,1}\right),\dots,\left(k-1\right)_{s}d\left(B_{s,k-1}\right)
\]
are pairwise disjoint sets as their $s^{th}$ coordinates disagree
pairwise. Hence we have 
\begin{equation}
\left|d\left(\bigcup_{t=0}^{k-1}t_{s}B_{s,t}\right)\right|=\left|\left(\bigcup_{t=1}^{k-1}0_{s}B_{s,t}\right)\cup0_{s}d\left(B_{s,0}\right)\right|+\sum_{t=1}^{k-1}\left|t_{s}d\left(B_{s,t}\right)\right|,\label{eq:3}
\end{equation}
and similarly we have 
\begin{equation}
\left|d\left(\bigcup_{t=0}^{k-1}t_{s}A_{s,t}\right)\right|=\left|\left(\bigcup_{t=1}^{k-1}0_{s}A_{s,t}\right)\cup0_{s}d\left(A_{s,0}\right)\right|+\sum_{t=1}^{k-1}\left|t_{s}d\left(A_{s,t}\right)\right|.\label{eq:4}
\end{equation}

Since Theorem 1 holds for $\left[k\right]^{n-1}$, it follows that
\begin{equation}
\left|t_{s}d\left(A_{s,t}\right)\right|=\left|d\left(A_{s,t}\right)\right|\geq\left|d\left(B_{s,t}\right)\right|=\left|t_{s}d\left(B_{s,t}\right)\right|.\label{eq:5}
\end{equation}
Note that the $d$-shadow of an initial segment  is also an initial
segment, and initial segments are nested. Hence we have 
\begin{equation}
\left|\left(\bigcup_{t=1}^{k-1}0_{s}B_{s,t}\right)\cup0_{s}d\left(B_{s,0}\right)\right|=\max\left(\left|B_{s,1}\right|,\dots,\left|B_{s,k-1}\right|,\left|d\left(B_{s,0}\right)\right|\right).\label{eq:6}
\end{equation}
Combining the trivial estimate

\begin{equation}
\left|\left(\bigcup_{t=1}^{k-1}0_{s}A_{s,t}\right)\cup0_{s}d\left(A_{s,0}\right)\right|\geq\max\left(\left|A_{s,1}\right|,\dots,\left|A_{s,k-1}\right|,\left|d\left(A_{s,0}\right)\right|\right)\label{eq:7}
\end{equation}
with (\ref{eq:5}), (\ref{eq:6}) and the fact that $\left|A_{s,i}\right|=\left|B_{s,i}\right|$
for all $0\leq i\leq k-1$, it follows that 
\begin{equation}
\left|\bigcup_{t=1}^{k-1}0_{s}A_{s,t}\cup0_{s}d\left(A_{s,0}\right)\right|\geq\left|\bigcup_{t=1}^{k-1}0_{s}B_{s,t}\cup0_{s}d\left(B_{s,0}\right)\right|.\label{eq:8}
\end{equation}
Thus pairing up the terms in (\ref{eq:3}) and (\ref{eq:4}) in the
natural way and applying (\ref{eq:5}) and (\ref{eq:8}) gives that
\begin{equation}
\left|d\left(A\right)\right|\geq\left|d\left(C_{s}\left(A\right)\right)\right|,\label{eq:9}
\end{equation}
which completes the proof. 
\end{proof}
We say that $T\subseteq\left[k\right]^{n}$ is \textit{compressed}
if $C_{s}\left(T\right)=T$ holds for all $s\in\left\{ 1,\dots,n\right\} $.
We now make the standard observation that it suffices to prove Theorem
1 for compressed sets. 
\begin{claim}
Let $A$ be a subset of $\left[k\right]^{n}$. Then there exists a
compressed set $B\subseteq\left[k\right]^{n}$ with $\left|B\right|=\left|A\right|$
and $\left|d\left(A\right)\right|\geq\left|d\left(B\right)\right|$.
\end{claim}

\begin{proof}
Consider a sequence $\left(A_{m}\right)$ with $A_{0}=A$ obtained
as follows. Given $A_{m}$, if there exists $s\in\left[n\right]$
for which $C_{s}\left(A_{m}\right)\neq A_{m}$ we set $A_{m+1}=C_{s}\left(A_{m}\right)$.
Otherwise we set $A_{m+1}=A_{m}$. 

Let $K_{i}$ be the $i^{th}$ set in $\left[k\right]^{n}$ with respect
to the $\leq$-order. Define $f\left(A\right)=\sum_{i=1}^{k^{n}}i\mathbb{I}\left\{ K_{i}\in A\right\} $,
where $\mathbb{I}\left\{ K_{i}\in A\right\} $ denotes the indicator
function of the event $K_{i}\in A$. By the construction of the compression
operator $C_{s}$, it is easy to verify that we have $f\left(C_{s}\left(A\right)\right)\leq f\left(A\right)$
for all $s\in\left\{ 1,\dots,n\right\} $, and for a given $s$ the
equality holds if and only if we have $C_{s}\left(A\right)=A$. Since
$f\left(A\right)$ is always a non-negative integer, it follows that
the sequence $f\left(A_{m}\right)$ is eventually constant. Thus there
exists $r$ for which we have $f\left(C_{s}\left(A_{r}\right)\right)=f\left(A_{r}\right)$
for all $s$, and hence we have $C_{s}\left(A_{r}\right)=A_{r}$ for
all $s$. Therefore $A_{r}$ is compressed. 

By Claim 1, if follows that we have $\left|d\left(A_{i}\right)\right|\geq\left|d\left(A_{i+1}\right)\right|$
for all $0\leq i\leq r-1$. Hence it follows that  $\left|d\left(A\right)\right|\geq\left|d\left(A_{r}\right)\right|$,
which completes the proof. 
\end{proof}
From now on, let $A\subseteq\left[k\right]^{n}$ denote a compressed
set and let $C$ denote the initial segment of $\left[k\right]^{n}$
of the same size. By Claim 2, it suffices to prove that $\left|d\left(A\right)\right|\geq\left|d\left(C\right)\right|$. 

\subsection{The special case $n=2$. }

Before moving on to the general case, we prove Theorem 1 when $n=2$.
This turns out to be convenient as $n=2$ is too small in one part
of the general argument. 
\begin{claim}
When $n=2$, Theorem 1 holds for a compressed set $A$. 
\end{claim}

\begin{proof}
The claim is trivial if $\left|A\right|=1$. Let $C$ denote the initial
segment of $\left[k\right]^{n}$ of size $\left|A\right|$. If $2\leq\left|A\right|\leq2k-1$,
then $C$ is a subset of $B_{\geq1}$ and hence it follows that $d\left(C\right)=\left\{ 00\right\} $.
Thus it is evident that $\left|d\left(A\right)\right|\geq\left|d\left(C\right)\right|$. 

Now suppose that $\left|A\right|\geq2k$. We write $A$ as $A=A_{0}\cup X$,
where $A_{0}=A\cap B_{\geq1}$ and 
\[
X=A\setminus A_{0}=\left\{ x_{1}x_{2}\in A:x_{1}\neq0\text{ and }x_{2}\neq0\right\} .
\]
Since $\left|B_{\geq1}\right|=2k-1$, it follows that $\left|X\right|\geq\left|A\right|-2k+1$,
and in particular $X$ is non-empty. 

Let $x_{1},\dots,x_{r}\in\left\{ 1,\dots,k-1\right\} $ and $y_{1},\dots,y_{s}\in\left\{ 1,\dots,k-1\right\} $
be chosen so that 
\[
d\left(X\right)=\left\{ 0x_{1},\dots,0x_{r},y_{1}0,\dots,y_{s}0\right\} .
\]
Then we certainly have
\[
X\subseteq\left\{ y_{j}x_{i}\,:\,1\leq j\leq s,\,1\leq i\leq r\right\} ,
\]
which implies that $\left|X\right|\leq rs$. Since for non-negative
integers $r$ and $s$ we have $r+s\geq\left\lceil \sqrt{4rs}\right\rceil $,
it follows that $\left|d\left(X\right)\right|\ge\left\lceil \sqrt{4\left|X\right|}\right\rceil $. 

Since $A$ is compressed and $\left|A\right|>1$, it follows that
$A$ contains a point $x_{1}x_{2}\ne00$ with $x_{1}=0$ or $x_{2}=0$.
In particular, we must have $00\in d\left(A\right)$. Hence it follows
that $d\left(A\right)=\left\{ 00\right\} \cup d\left(X\right)$. As
$\left|X\right|\geq\left|A\right|-2k+1$, we have 
\[
\left|d\left(A\right)\right|\geq1+\left\lceil \sqrt{4\left(\left|A\right|-2k+1\right)}\right\rceil .
\]

If $r^{2}+2k\leq\left|A\right|\leq r^{2}+r+2k-1$ for some $1\leq r\leq k-2$,
it is easy to verify that $C$ satisfies $C\subseteq B_{\geq1}\cup\left\{ x_{1}x_{2}:1\leq x_{1}\leq r+1,\,1\leq x_{2}\leq r\right\} $.
In particular, we have 

\[
\left|d\left(C\right)\right|\leq1+\left(r+1\right)+r=1+\left\lceil \sqrt{4\left(\left|A\right|-2k+1\right)}\right\rceil ,
\]
and hence it follows that $\left|d\left(A\right)\right|\geq\left|d\left(C\right)\right|$.
If $r^{2}+r+2k\leq\left|A\right|\leq\left(r+1\right)^{2}+2k-1$ for
some $1\leq r\leq k-2$, it is easy to verify that $C$ satisfies
$C\subseteq B_{\geq1}\cup\left\{ x_{1}x_{2}:1\leq x_{1},x_{2}\leq r+1\right\} $.
Again, we have 
\[
\left|d\left(C\right)\right|\leq1+2\left(r+1\right)=1+\left\lceil \sqrt{4\left(\left|A\right|-2k+1\right)}\right\rceil .
\]
Thus it follows that $\left|d\left(A\right)\right|\geq\left|d\left(C\right)\right|$,
which completes the proof. 
\end{proof}

\subsection{General observations}

From now on we suppose that $n\geq3$, and our aim is to prove Theorem
1 for a compressed set $A\subseteq\left[k\right]^{n}$. As in the
previous section, the proof is trivial if $\left|A\right|=1$ or $2\leq\left|A\right|\leq n\left(k-1\right)+1$,
as in these cases the $d$-shadow of the initial segment has size
$0$ or $1$ respectively. Thus we may assume that $\left|A\right|\geq n\left(k-1\right)+2$. 

We say that $B\subseteq\left[k\right]^{n}$ is a \textit{down-set}
if for any $y\in B$ and for a point $x\in\left[k\right]^{n}$ satisfying
$x_{j}\leq y_{j}$ for all $j$ we have $x\in B$. In this subsection
we make some observations on the structure of compressed set $A$.
We start by proving that $A$ is a down-set, and in fact this follows
from a slightly stronger statement. Secondly, recall that for an initial
segment $C$ there exists $r$ satisfying $B_{\geq r}\subseteq C\subseteq B_{\geq r+1}$.
For compressed sets the same conclusion does not necessarily hold,
but as a second structural claim we prove that for compressed sets
there exists $r$ for which $d\left(B_{\geq r}\right)\subseteq d\left(A\right)\subseteq d\left(B_{\geq r+1}\right)$.
Since we only consider the size of the $d$-shadow of $A$, in a sense
this is 'equally good' for our purposes as having $B_{\geq r}\subseteq A\subseteq B_{\geq r+1}$.
Note that the second observation allows us to focus only on $d\left(A\cap B_{r+1}\right)$. 

Before proving that $A$ and $d\left(A\right)$ are down-sets, we
start by a straightforward observation that turns out to be very 
useful. Given points $x,y$ with $x\leq y$ and $y\in A$, suppose
that there exists $i$ for which $x_{i}=y_{i}$, and for convenience
set $t=x_{i}$. Since $A$ is compressed, it follows that $A_{i,t}$
is an initial segment. Let $x'$ and $y'$ be the points obtained
from $x$ and $y$ by removing the $i^{th}$ coordinate, and note
that $y'\in A_{i,t}$. Since $x\leq y$ and $x_{i}=y_{i}$, it follows
that $x'\leq y'$. Since $A$ is compressed and $y'\in A_{i,t}$,
it follows that $x'\in A_{i,t}$ and hence we also have $x\in A$.
In summary, we have proved that 

\begin{equation}
x\leq y,\,y\in A,\,x_{i}=y_{i}\text{ for some }i\Rightarrow x\in A.\label{eq:10-1}
\end{equation}
This turns out to be a very useful fact that we will use throughout
the rest of the proof. As a simple consequence, we now obtain that
$A$ and $d\left(A\right)$ are down-sets. 
\begin{claim}
Let $A\subseteq\left[k\right]^{n}$ be a compressed set. Then $A$
is a down-set, and furthermore $d\left(A\right)$ is also a down-set. 
\end{claim}

\begin{proof}
Let $y\in A$ and let $x\in\left[k\right]^{n}$ be a point satisfying
$x_{i}\leq y_{i}$ for all $i$. Let $z\in\left[k\right]^{n}$ be
obtained by taking $z_{1}=x_{1}$ and $z_{s}=y_{s}$ for all $s\geq2$,
and note that we certainly have $x\leq z\leq y$. Since $y_{2}=z_{2}$
and $x_{1}=z_{1}$, two applications of (\ref{eq:10-1}) implies that
$z\in A$ and $x\in A$, which completes the proof of the first part.

Let $y\in d\left(A\right)$ and let $x\in\left[k\right]^{n}$ for
which $x_{i}\leq y_{i}$ for all $i$. Choose $v\in A$ satisfying
$y\in d\left(v\right)$, and let $a$ be the unique index for which
$v_{a}\neq0$ but $y_{a}=0$. Let $u$ be the point obtained by setting
$u_{a}=v_{a}$ and $u_{j}=x_{j}$ for all $j\neq a$. Then for $j\neq a$
we have $u_{j}=x_{j}\leq y_{j}=v_{j}$. Since we also have $u_{a}=v_{a}$,
the first part implies that $u\in A$. Since $x_{a}\leq y_{a}=0$,
we must have $x_{a}=0$. Hence we have $x\in d\left(u\right)\subseteq d\left(A\right)$,
which completes the proof. 
\end{proof}
Before moving on to the second structural fact, we introduce some
notation that will be used throughout the rest of the paper. For $x=x_{1}\dots x_{n}\in\mathbb{N}^{n}$,
set $m\left(x\right)=\max\left(x_{1},\dots,x_{n}\right)$ and recall
that for all $i\in\mathbb{N}$ we defined $R_{i}\left(x\right)=\left\{ j:x_{j}=i\right\} $.
Define $c\left(x\right)=\left(m\left(x\right),R_{m\left(x\right)}\left(x\right),\left|R_{0}\left(x\right)\right|\right)$.
That is, the first coordinate of $c\left(x\right)$ is $\max\left(x_{1},\dots,x_{n}\right)$,
the second coordinate is the set of all positions where this maximum
is attained and the last coordinate is the number of $x_{i}$'s that
equal 0. Define the \textit{component} of $x$ by $C_{x}=\left\{ y\in\left[k\right]^{n}:c\left(y\right)=c\left(x\right)\right\} $. 

Note that we have $x\in\left[k\right]^{n}$ if and only if $C_{x}\subseteq\left[k\right]^{n}$,
as every $y\in C_{x}$ has the same maximum value of a coordinate.
Moreover, since every $y\in C_{x}$ also shares the same positions
of the maximum coordinates and same number of coordinates that equal
0, it follows that for any other class $C_{z}$ either all the points
of $C_{x}$ occur before the points of $C_{z}$ with respect to the
$\leq$-order, or all the points of $C_{x}$ occur after the points
of $C_{z}$. Hence we can order the classes inside $\left[k\right]^{n}$
as $C_{1},\dots,C_{m}$ such that for any $i\neq j$, $x\in C_{i}$
and $y\in C_{j}$ we have $x\leq y$ if and only if $i<j$.

Note that for all $x,y\in C_{i}$ there exists $r$ such that $x_{r}=y_{r}$.
Indeed, any $r\in R_{m\left(x\right)}\left(x\right)=R_{m\left(y\right)}\left(y\right)$
works. Hence (\ref{eq:10-1}) implies that for any $i$ we have $A\cap C_{i}=\emptyset$
or there exists $y_{i}\in C_{i}$ such that $A\cap C_{i}=\left\{ x\in C_{i}\,:\,x\leq y_{i}\right\} $. 

For fixed $s$ and $t$, the classes of the form $\left(s,\,A,\,t\right)$
for $A\in\left[n\right]^{\left(\leq n-t\right)}$ occur consecutively
with respect to the $\leq$-order. Furthermore, from the definition
of the $\leq$-order it is easy to verify that these classes occur
in the order induced by the binary order on $\left[n\right]^{\left(\leq n-t\right)}$.
In particular, if $A_{i}$ and $A_{i+1}$ are two consecutive sets
in binary order with $\left|A_{i}\right|\leq n-t$ and $\left|A_{i+1}\right|\leq n-t$,
then $\left(s,\,A_{i},t\right)$ immediately precedes $\left(s,\,A_{i+1},\,t\right)$
in the order of classes. 

Recall that $B_{s}$ is defined to be the set of those points with
exactly $s$ coordinates that equal 0, i.e. 
\[
B_{s}=\left\{ x\in\left[k\right]^{n}:\left|R_{0}\left(x\right)\right|=s\right\} ,
\]
and $B_{\geq s}=\bigcup_{i=s}^{n}B_{i}$. Note that if $X$ is an
initial segment of the $\leq$-order, there exists $r$ such that
$B_{\geq r+1}\subseteq X\subset B_{\geq r}$. Our next aim is to prove
that for any compressed set $A$ there exists $r$ so that $d\left(B_{r}\right)\subseteq d\left(A\right)\subseteq d\left(B_{r+1}\right)$. 
\begin{claim}
Let $A$ be a compressed set, and let $0\leq p\leq n$ be the minimal
index for which $A\cap B_{p}\neq\emptyset$. Then we have $d\left(B_{\geq p+1}\right)\subseteq d\left(A\right)\subseteq d\left(B_{p}\right)$.
In particular, if $r$ is chosen such that $\left|B_{\geq r+1}\right|<\left|A\right|\leq\left|B_{\geq r}\right|$,
it suffices to prove Theorem 1 for compressed sets $A$ that satisfy
$B_{\ge r+1}\subseteq A\subseteq B_{\geq r}$. 
\end{claim}

\begin{proof}
Let $p$ be the minimal index for which $A\cap B_{p}\neq\emptyset$,
and let $u$ be the minimal point under the $\leq$-order in $B_{p}$.
Since $A$ is a down-set and $u\in B_{p}$, it follows that every
coordinate of $u$ must be either $0$ or $1$. Hence there exists
a set $X$ of size $n-p$ so that $u_{i}=\mathbb{I}\left\{ i\in X\right\} $. 

Let $x\in B_{\geq p+2}=d\left(B_{\geq p+1}\right)$. If $R_{0}\left(x\right)\cap X\neq\emptyset$,
choose $i\in R_{0}\left(x\right)\cap X$ and consider the point $y$
obtained by taking $y_{j}=x_{j}$ for $j\neq i$ and $y_{i}=1$. Since
$i\in X$, it follows that $u_{i}=1=y_{i}$. On the other hand, since
$\left|R_{0}\left(y\right)\right|=p+2-1=p+1>p=\left|R_{0}\left(u\right)\right|$
it follows that $y\leq u$. Hence (\ref{eq:10-1}) implies that $y\in A$,
and thus we have $x\in d\left(A\right)$. 

If $R_{0}\left(x\right)\cap X=\emptyset$, choose any $i\in R_{0}\left(x\right)$
and again consider the point  $y$ obtained by taking $y_{j}=x_{j}$
for $j\ne i$ and $y_{i}=1$. Let $j\in R_{0}\left(x\right)\setminus\left\{ i\right\} $,
and note that such $j$ exists as $\left|R_{0}\left(x\right)\right|\geq p+2\geq2$.
As $j\not\in X$, we have $u_{j}=0=x_{j}=y_{j}$. Similarly as in
the first case, we have $\left|R_{0}\left(y\right)\right|>\left|R_{0}\left(u\right)\right|$
and thus it follows that $y\leq u$. Hence (\ref{eq:10-1}) implies
that $y\in A$, and therefore we have $x\in d\left(A\right)$, which
completes the proof of the first part. 

For the second part, suppose that $A'$ is a compressed set satisfying
$\left|B_{\geq r+1}\right|<\left|A'\right|\leq\left|B_{\geq r}\right|$
for some $r$. Let $C$ denote the initial segment of the $\leq$-order
with $\left|C\right|=\left|A'\right|$, and assume that Theorem 1
holds for any compressed set $A\subseteq\left[k\right]^{n}$ satisfying
$\left|A\right|=\left|A'\right|$ and $B_{\geq r+1}\subseteq A\subseteq B_{\geq r}$.
Let $p$ denote the least integer for which $A'\cap B_{p}\neq\emptyset$. 
Since $\left|A'\right|>\left|B_{\geq r+1}\right|$, it follows that
$p\leq r$. 

If $p<r$, the first part implies that $d\left(B_{\ge r}\right)\subseteq d\left(B_{\geq p+1}\right)\subseteq d\left(A'\right)$,
and since $C\subseteq B_{\ge r}$ it certainly follows that $\left|d\left(A'\right)\right|\geq\left|d\left(C\right)\right|$.
If $p=r$, the choice of $p$ implies that $A'\subseteq B_{\geq r}$,
and the first part implies that $d\left(B_{\geq r+1}\right)\subseteq d\left(A'\right)$.
Let $D'=A'\cap B_{r}$, and note that the last two observations imply
that $d\left(A'\right)=d\left(B_{\geq r+1}\right)\cup d\left(D'\right)$.
Let $C'=C\cap B_{r}$, and note that since $\left|A'\right|=\left|C\right|$
and $B_{\geq r+1}\subseteq C$, it follows that $\left|C'\right|\leq\left|D'\right|$.
Let $D$ be the set of $\left|C'\right|$ smallest points of $D'$
under the $\leq$-order, and let $A=B_{\geq r+1}\cup D$. It is easy
to see that $\left|A\right|=\left|C\right|=\left|A'\right|$ and that
$A$ is compressed, and by the construction of $A$ it follows that
$B_{\geq r+1}\subseteq A\subseteq B_{\geq r}$. Furthermore, since
$D\subseteq D'$ and $d\left(B_{\geq r+1}\right)\subseteq d\left(A'\right)$,
it follows that $d\left(A\right)\subseteq d\left(A'\right)$. In particular,
we must have $\left|d\left(A'\right)\right|\geq\left|d\left(A\right)\right|$,
which completes the proof. 
\end{proof}

\subsection{Proof of Theorem 1 when $n\geq3$. }

Let $A$ be a compressed set, let $r$ be chosen such that $\left|B_{\ge r+1}\right|<\left|A\right|\leq\left|B_{\ge r}\right|$
and let $C$ be the the initial segment of the $\leq$-order with $\left|C\right|=\left|A\right|$.
Recall that by Claim 5 we may assume that we have $B_{\geq r+1}\subseteq A\subseteq B_{\geq r}$.
From now on we set $D=A\cap B_{r}$. Since $\left|d\left(A\right)\right|=\left|B_{\geq r+2}\right|+\left|d\left(D\right)\right|$
and $\left|d\left(C\right)\right|=\left|B_{\geq r+2}\right|+\left|d\left(C\cap B_{r}\right)\right|$,
it suffices to focus on comparing $d\left(D\right)$ with $d\left(C\cap B_{r}\right)$. 

We split the proof into two cases depending on whether $\left|A\right|\leq\left|B_{\ge1}\right|$
or $\left|A\right|>\left|B_{\geq1}\right|$. 

\subsection*{Case 1. $\left|A\right|\protect\leq\left|B_{\ge1}\right|.$}

Since $\left|A\right|\leq\left|B_{\geq1}\right|$, it follows that
$A=B_{\geq r+1}\cup D$ for some $D\subseteq B_{r}$ and $r\geq1$.
Recall that Theorem 1 is trivial when $\left|A\right|\leq n\left(k-1\right)+1$,
and hence we may assume that $r\leq n-2$. 

Note that each of the sets $B_{\geq r+1}\cup\left(\left[s\right]^{n}\cap B_{r}\right)$
is an initial segment of the $\leq$-order, and recall that $\left[s\right]^{n}=\left\{ 0,\dots,s-1\right\} ^{n}$.
We start by comparing the sets $d\left(D\right)$ and $d\left(\left[s\right]^{n}\cap B_{r}\right)$
for an appropriately chosen $s$. 
\begin{claim}
Let $s=\max\left\{ m\left(x\right):x\in D\right\} =\max\left\{ x_{i}:x_{1}\dots x_{n}\in D\right\} $.
Then we have $d\left(\left[s\right]^{n}\cap B_{r}\right)\subseteq d\left(A\right)$.
In particular, if $r$ and $s$ are chosen so that 
\[
\left|B_{\geq r+1}\cup\left(\left[s\right]^{n}\cap B_{r}\right)\right|<\left|A\right|\leq\left|B_{\geq r+1}\cup\left(\left[s+1\right]^{n}\cap B_{r}\right)\right|,
\]
then it suffices to prove Theorem 1 for compressed sets $A$ that
satisfy 
\[
B_{\ge r+1}\cup\left(\left[s\right]^{n}\cap B_{r}\right)\subseteq A\subseteq B_{\geq r}\cup\left(\left[\left(\left[s+1\right]^{n}\cap B_{r}\right)\right]\right).
\]
\end{claim}

\begin{proof}
When $s=1$ the claim is equivalent to the condition $B_{\geq r+1}\subseteq d\left(A\right)$
which is certainly true. Now suppose that $s\geq2$, and let $x$
be the least point under the $\leq$-order in $D$ with $m\left(x\right)=s$.
Since $A$ is a down-set, the minimality of $x$ implies that there
exists a unique $i$ satisfying $x_{i}=s$.

Let $v\in d\left(\left[s\right]^{n}\cap B_{r}\right)$. Since $\left|R_{0}\left(v\right)\right|=r+1\ge2$,
it follows that there exists $j\neq i$ with $v_{j}=0$. If $x_{j}\neq0$,
consider the point $z$ obtained by setting

\[
z_{t}=\begin{cases}
\begin{array}{cc}
v_{t} & \text{if }t\neq j\\
x_{j} & \text{if }t=j
\end{array}.\end{cases}
\]
Otherwise, pick any $p\neq j$ with $v_{p}=0$, and consider $z$
obtained by setting 

\[
z_{t}=\begin{cases}
\begin{array}{cc}
v_{t} & \text{if }t\neq p\\
1 & \text{if }t=p
\end{array}.\end{cases}
\]

Since $v_{t}\in\left[s\right]$ for all $t$ and $j$ is chosen so
that $x_{j}\neq s$, it follows that in either case we have $z\in\left[s\right]^{n}$.
In both cases we certainly have $v\in d\left(z\right)$ by the construction
of $z$. Since $v\in d\left(\left[s\right]^{n}\cap B_{r}\right)=\left[s\right]^{n}\cap B_{r+1}$,
it follows that $\left|R_{0}\left(x\right)\right|=\left|R_{0}\left(z\right)\right|$.
Since $m\left(x\right)>m\left(z\right)$, we have $z\leq x$. 

In the first case we have $x_{j}=z_{j}$, and hence (\ref{eq:10-1})
implies that $z\in A$. In the second case we have $z_{j}=v_{j}=x_{j}=0$,
so again (\ref{eq:10-1}) implies that $z\in A$. Hence it follows
that $v\in d\left(A\right)$, which completes the proof of the first
part. 

For the second part, suppose that $A'$ is a compressed set satisfying
$\left|B_{\geq r+1}\cup\left(\left[s\right]^{n}\cap B_{r}\right)\right|<\left|A'\right|\leq\left|B_{\geq r+1}\cup\left(\left[s+1\right]^{n}\cap B_{r}\right)\right|$
for some $r$ and $s$. Recall that by Claim 5 we may assume that
$B_{\geq r+1}\subseteq A'\subseteq B_{\geq r}$. Let $C$ be the initial
segment of the $\leq$-order with $\left|C\right|=\left|A'\right|$, and note
that $C\subseteq B_{\geq r+1}\cup\left(\left[s+1\right]^{n}\cap B_{r}\right)$. 

Let $t=\max\left(m\left(x\right):x\in A'\cap B_{r}\right)$. Since
$\left|A'\right|>\left|B_{\geq r+1}\cup\left(\left[s\right]^{n}\cap B_{r}\right)\right|$,
we must have $t\geq s$ (recall that $\left[s\right]^{n}=\left\{ 0,\dots,s-1\right\} ^{n}$).
If $t\geq s+1$, the first part implies that $d\left(B_{\geq r+1}\cup\left(\left[s+1\right]^{n}\cap B_{r}\right)\right)\subseteq d\left(A'\right)$.
In particular, it follows that $d\left(C\right)\subseteq d\left(A'\right)$
and thus $\left|d\left(A'\right)\right|\geq\left|d\left(C\right)\right|$,
as required. 

If $t=s$, it follows that $A'\subseteq B_{\geq r+1}\cup\left(\left[s+1\right]^{n}\cap B_{r}\right)$.
Let $X=\left\{ x\in A\cap B_{r}:m\left(x\right)=s\right\} $, and
note that $X\neq\emptyset$. Let $Y$ be the set of $k$ least points
of $X$ under the $\leq$-order, where $k=\left|A'\right|-\left|B_{\geq r+1}\cup\left(\left[s\right]^{n}\cap B_{r}\right)\right|$,
and consider $A$ defined by $A=B_{\geq r+1}\cup\left(\left[s\right]^{n}\cap B_{r}\right)\cup Y$.
It is clear that $A$ is compressed, and we have $\left|A\right|=\left|A'\right|=\left|C\right|$.

By the construction of $A$ it follows that $B_{\geq r+1}\cup\left(\left[s\right]^{n}\cap B_{r}\right)\subseteq A\subseteq B_{\geq r+1}\cup\left(\left[s+1\right]^{n}\cap B_{r}\right)$.
By Claim 5 and the first part, we have that $d\left(B_{\geq r+1}\cup\left(\left[s\right]^{n}\cap B_{r}\right)\right)\subseteq d\left(A'\right)$.
In particular, it follows that $d\left(B_{\geq r+1}\cup\left(\left[s\right]^{n}\cap B_{r}\right)\right)\cup d\left(X\right)\subseteq d\left(A'\right)$.
Since $d\left(Y\right)\subseteq d\left(X\right)$, we must have $d\left(A\right)=d\left(B_{\geq r+1}\cup\left(\left[s\right]^{n}\cap B_{r}\right)\right)\cup d\left(Y\right)\subseteq d\left(A'\right)$.
Hence it follows that $\left|d\left(A'\right)\right|\geq\left|d\left(A\right)\right|$,
which completes the proof. 
\end{proof}
Let $s=\max\left(m\left(x\right):x\in D\right)$. From now on we suppose
that $A$ is a compressed set that satisfies $B_{\geq r+1}\cup\left(\left[s\right]^{n}\cap B_{r}\right)\subset A\subseteq B_{\geq r+1}\cup\left(\left[s+1\right]^{n}\cap B_{r}\right)$.

We split our proof into two subcases based on whether $s=\max\left(m\left(x\right):x\in D\right)$
satisfies $s=1$ or $s\geq2$. When $s=1$, Theorem 1 is actually
equivalent to the Kruskal-Katona Theorem, and hence the case $s\geq2$
is the main part of the proof. 

\subsubsection*{Case 1.1. $s=1.$}

Since $s=1$, it follows that $D\subseteq\left\{ 0,1\right\} ^{n}\cap B_{r}$.
However, $\left\{ 0,1\right\} ^{n}\cap B_{r}$ is just the set of
those points in $\left\{ 0,1\right\} ^{n}$ with exactly $r$ coordinates
that equal $0$, i.e.\ the set $\left\{ 0,1\right\} _{n-r}^{n}$. It
is easy to check that the colexicographic order and the $\leq$-order
coincide on $\left\{ 0,1\right\} _{n-r}^{n}$. Thus Theorem 1 holds
when $s=1$, as the Kruskal-Katona theorem states that the initial
segments of the colexicographic order minimise the lower shadow on
$\left\{ 0,1\right\} _{n-r}^{n}$.

\subsubsection*{Case 1.2. $s\protect\geq2.$}

Since $s\ge2$, it follows that for any point $x$ changing a value
of a non-zero coordinate to 1 cannot increase the set $R_{s}\left(x\right)$.
As a consequence of Claim 6, from now on we only need to focus on
those points in $D$ that contain  $s$ as a coordinate. 

Let ${\cal A}=\left[n\right]^{\left(\leq n-r\right)}\setminus\left\{ \emptyset\right\} $,
and for $X\in{\cal A}$ denote the class $C_{i}$ corresponding to
$\left(s,X,r\right)$ by $C_{X}$. Note that ${\cal A}$ characterizes
all the classes containing points with $s$ as a maximal coordinate
and $r$ zeroes. As noted before, the order of classes under the $\leq$-order
is the order induced on ${\cal A}$ by the binary order.

Let $T\in{\cal A}$ be the largest set under the binary order with
$C_{T}\cap D\neq\emptyset$. We start by proving that for any $S\in{\cal A}$
with $S<_{bin}T$ we have $d\left(C_{S}\right)\subseteq d\left(D\right)$.
If $S\cap T\neq\emptyset$, the claim follows easily as $C_{S}\subseteq D$
holds by (\ref{eq:10-1}). We start by proving the claim when $T$
contains only one element, as in this case such an easy argument does
not exist and the stronger conclusion $C_{s}\subseteq D$ may not
always hold. In Claim 8 we deduce that for any $S\in{\cal A}$ with
$S<_{bin}T$ we have $d\left(C_{S}\right)\subseteq d\left(D\right)$
regardless of the size of $T$. 
\begin{claim}
Let $X\in{\cal A}$ be a set satisfying $C_{X}\cap D\neq\emptyset$
and $\left|X\right|=1$. Then for any $S\in{\cal A}$ with $S<_{bin}X$
we have $d\left(C_{S}\right)\subseteq d\left(D\right)$. 
\end{claim}

\begin{proof}
Let $i$ be chosen so that $X=\left\{ i\right\} $, and define the
particular point  $a=s_{i}\left(\left(\left(n-r-1\right)\cdot1\right)\left(r\cdot0\right)\right)$.
Note that $a$ is the least point in $C_{X}$, and hence (\ref{eq:10-1})
implies that $a\in D$. 

Let $x\in d\left(C_{S}\right)$. Since $r\ge1$, it follows that $\left|R_{0}\left(x\right)\right|=r+1\ge2$.
Hence there exists distinct elements $l$ and $m$ with $x_{l}=x_{m}=0$.
In particular, we may assume that $m\neq i$. 

Let $y$ be the point obtained by taking $y_{j}=x_{j}$ for $j\neq m$
and $y_{m}=1$, and $z$ be the point obtained by taking $z_{j}=x_{j}$
for $j\neq l$ and $z_{l}=1$. Note that we have $R_{s}\left(y\right)\subseteq S$
and $R_{s}\left(z\right)\subseteq S$, and recall that $R_{s}\left(a\right)=X$. 
Since $S<_{bin}X$, these conditions imply that $y\le a$ and $z\leq a$.
Note that by the construction of $y$ and $z$ we have $y_{m}=1$
and $z_{m}=x_{m}=0$. Since $m\neq i$, it follows that $a_{m}\in\left\{ 0,1\right\} $
and hence we either have $a_{m}=y_{m}$ or $a_{m}=z_{m}$. Thus (\ref{eq:10-1})
implies that we have $y\in D$ or $z\in D$, and hence in either case
it follows that $x\in d\left(D\right)$, which completes the proof. 
\end{proof}
\begin{claim}
Let $T\in{\cal A}$ be the largest element under the binary order with
$C_{T}\cap D\neq\emptyset$. Then for any $S\in{\cal A}$ with $S<_{bin}T$
we have $d\left(C_{S}\right)\subseteq d\left(D\right)$. 
\end{claim}

\begin{proof}
If $\left|T\right|=1$, the result follows immediately from Claim
7. Hence we may assume that $\left|T\right|>1$, and define $T_{1}=\left\{ \max T\right\} $.
Note that $T_{1}\cap T\neq\emptyset$ and $T_{1}<_{bin}T$. Hence
(\ref{eq:10-1}) implies that $C_{T_{1}}\subseteq D$.

Let $S\in{\cal A}$ for which $S<_{bin}T$. If $\max S=\max T$, then
(\ref{eq:10-1}) implies that $C_{S}\subseteq D$, and in particular
it follows that $d\left(C_{S}\right)\subseteq d\left(D\right)$. If
$\max S<\max T$, it follows that $S<_{bin}T_{1}$. By applying Claim
7 for the set $T_{1}$ which contains only one point, it follows that
$d\left(C_{S}\right)\subseteq d\left(D\right)$, which completes the
proof. 
\end{proof}
Now we are ready to finish the proof of Case 1.2. Let $C$ be the
initial segment of the $\leq$-order with $\left|C\right|=\left|A\right|$.
Recall that $r$ and $s$ are chosen so that $B_{\geq r+1}\cup\left(\left[s\right]^{n}\cap B_{r}\right)\subseteq A\subseteq B_{\geq r+1}\cup\left(\left[s+1\right]^{n}\cap B_{r}\right)$.
Since $C$ is an initial segment, it follows that there exists $T_{1}\in{\cal A}$
and $x\in T_{1}$ so that 
\[
C=B_{\geq r+1}\cup\left(\left[s\right]^{n}\cap B_{r}\right)\cup\left(\bigcup_{S\in{\cal A},\,S<_{bin}T_{1}}C_{S}\right)\cup\left\{ y:y\leq x,y\in T_{1}\right\} .
\]

As $\left|A\right|=\left|C\right|$, it follows that $T_{1}\le_{bin}T$.
If $T_{1}<T$, we have $d\left(C_{S}\right)\subseteq d\left(A\right)$
for all $S$ with $S\in{\cal A}$ and $S\leq_{bin}T_{1}$ by Claim
8. In particular, it follows that $d\left(C\right)\subseteq d\left(A\right)$,
and therefore $\left|d\left(A\right)\right|\geq\left|d\left(C\right)\right|$.
If $T_{1}=T$, then we have $A\cap C_{T}=\left\{ y:y\leq z,y\in C_{T}\right\} $
for some $z\in C_{T}$ with $x\leq z$ as $C$ is an initial segment.
Hence we have $d\left(\left\{ y:y\leq x,y\in T_{1}\right\} \right)\subseteq d\left(A\right)$,
and it follows that $d\left(C\right)\subseteq d\left(A\right)$. This
completes the proof of Case 1.2. 

\subsection*{Case 2. $\left|A\right|>\left|B_{\protect\geq1}\right|.$}

Recall that $A$ is of the form $A=B_{\geq1}\cup D$ for some $D\subseteq B_{0}=\left\{ 1,\dots,k-1\right\} ^{n}$.
Note that Theorem 1 holds when $\left|A\right|=\left|B_{\geq1}\right|+1$,
as $\left|d\left(x\right)\right|=n$ for any $x\in B_{0}$. From now
on we assume that $\left|A\right|>\left|B_{\geq1}\right|+1$. Hence
it follows that $\left|D\right|>1$, and thus $s=\max\left(m\left(x\right):x\in D\right)$
satisfies $s\geq2$.

Ideally, we would like to use an approach that is similar to the one used in Case 1. However, it turns out that
the statements of Claims 6 and 8 need to be slightly modified. 
We start with a preliminary result which states that (\ref{eq:10-1})
also holds for points inside $d\left(D\right)$. Then we move on to
prove an appropriate version of Claim 6. When $T\neq\left\{ 1\right\} $
it follows that we again have $d\left(\left\{ 1,\dots,s-1\right\} ^{n}\right)\subseteq d\left(A\right)$,
but when $T=\left\{ 1\right\} $ one specific point might be missing
from the shadow. 
\begin{claim}
Let $y\in d\left(D\right)$, and let $x$ be a point satisfying $\left|R_{0}\left(x\right)\right|=1$
and $x\leq y$. Furthermore, suppose that there exists $i$ for which
$x_{i}=y_{i}$. Then we have $x\in d\left(D\right)$. 
\end{claim}

\begin{proof}
We may certainly assume that $x<y$. Since $y\in d\left(D\right)$,
there exists $v\in D$ with $y\in d\left(v\right)$. Thus there exists
$b$ satisfying $v_{j}=y_{j}$ for $j\neq b$ and $y_{b}=0$, and
since $A$ is a down-set, we may assume that $v_{b}=1$. Note that
$b$ is the unique index for which $y_{b}=0$ as $\left|R_{0}\left(y\right)\right|=1$. 

Since $\left|R_{0}\left(x\right)\right|=1$, there exists a unique
index $a$ with $x_{a}=0$. Let $u$ be the sequence obtained by taking
$u_{j}=x_{j}$ for $j\neq a$ and $u_{a}=1$. Note that we have $x\in d\left(u\right)$. 

Our aim is to prove that $u\in D$. Let $t$ be the largest index
for which $R_{t}\left(x\right)\neq R_{t}\left(y\right)$, and for
such $t$ we have $\max\left(R_{t}\left(x\right)\Delta R_{t}\left(y\right)\right)\in R_{t}\left(y\right)$.
Observe that we must have $t\geq1$, as for any point $z$ the sets
$R_{t}\left(z\right)$ are disjoint and their union equals $\left\{ 1,\dots,n\right\} $. 

First note that by the construction of the points $u$ and $v$ it
follows that $R_{s}\left(u\right)=R_{s}\left(x\right)$ and $R_{s}\left(v\right)=R_{s}\left(y\right)$
for $s\geq2$. In particular, it follows that $R_{s}\left(u\right)=R_{s}$$\left(v\right)$
for $s>t$. 

We start by proving that $t=1$ implies $u=v$. Indeed, by the previous
observation we have $R_{s}\left(u\right)=R_{s}\left(v\right)$ for
all $s>1$, and we also have $R_{0}\left(u\right)=R_{0}\left(v\right)=\emptyset$.
Since for any point $c$ the sets $R_{s}\left(c\right)$ are disjoint
and their union equals $\left\{ 1,\dots,n\right\} $, we must also
have $R_{1}\left(u\right)=R_{1}\left(v\right)$. In particular, it
follows that $u=v$. 

If $t\geq2$, we have $R_{t}\left(u\right)=R_{t}\left(x\right)$ and
$R_{t}\left(v\right)=R_{t}\left(y\right)$. Hence $\max\left(R_{t}\left(u\right)\Delta R_{t}\left(v\right)\right)\in R_{t}\left(v\right)$,
and thus it follows that $u<v$. Recall that there exists $i$ satisfying
$x_{i}=y_{i}$. If $i\not\in\left\{ a,b\right\} $, it follows that
$u_{i}=x_{i}$ and $v_{i}=y_{i}$ by the construction of the points
$u$ and $v$. In particular, we have $u_{i}=v_{i}$, and thus (\ref{eq:10-1})
implies that $u\in D$. 

Now suppose that $i\in\left\{ a,b\right\} $. Since $x_{a}=0$ and
$y_{b}=0$, we must have $x_{i}=y_{i}=0$. However, recall that both
$R_{0}\left(x\right)$ and $R_{0}\left(y\right)$ contain only one
point. Hence we must have $a=b=i$, and thus by the construction of
$u$ and $v$ it follows that $u_{i}=v_{i}=1$. As before, (\ref{eq:10-1})
implies that $u\in D$, which completes the proof as $x\in d\left(u\right)$. 
\end{proof}
~ \\
\begin{claim}
~
\begin{enumerate}
\item If $T\neq\left\{ 1\right\}$, we have $d\left(\left\{ 1,\dots,s-1\right\} ^{n}\right)\subseteq d\left(A\right)$.
\item If $T=\left\{ 1\right\}$, we have \textbf{$d\left(\left\{ 1,\dots,s-1\right\} ^{n}\right)\setminus\left\{ 0\left(\left(n-1\right)\cdot\left(s-1\right)\right)\right\} \subseteq d\left(A\right)$}.
\end{enumerate}
\end{claim}

\begin{proof}
We start with the case $T\neq\left\{ 1\right\} $, and recall that
$s\geq2$. Let $j=\max\left(T\right)$, which by assumption is at
least 2. Define the particular point $a=s_{j}\left(\left(n-1\right)\cdot1\right)$
which is the least point in $C_{\left\{ j\right\} }$. Let $z$ be
a point in $C_{T}\cap D$. Since $a\leq z$ and $a_{j}=z_{j}$, (\ref{eq:10-1})
implies that $a\in D$. Let $b=s_{j-1}\left(\left(n-1\right)\cdot1\right)$,
and note that $b$ is well-defined as $j>1$. Since $2\left(n-1\right)>n$,
there exists $i$ for which $a_{i}=b_{i}=1$ by the pigeonhole principle.
Thus (\ref{eq:10-1}) implies that $b\in D$. 

Let $x\in d\left(\left\{ 1,\dots,s-1\right\} \right)$, and let $i$
be the unique index satisfying $x_{i}=0$. Let $y$ be obtained by
taking $y_{r}=x_{r}$ for $r\neq i$ and $y_{i}=1$. Since $s\geq2$,
it follows that $y\in\left\{ 1,\dots,s-1\right\} ^{n}$, so we have
$y\leq a$ and $y\leq b$. If $i\neq j$ we have $a_{i}=y_{i}=1$,
and if $i=j$ we have $b_{i}=y_{i}=1$. Since $a,b\in D$, we must
have $y\in D$ in either case by (\ref{eq:10-1}). This completes
the proof of the first part. 

Now suppose that $T=\left\{ 1\right\} $. For convenience define the
points $w=0\left(\left(n-1\right)\cdot\left(s-1\right)\right)$ and
$a=s\left(\left(n-1\right)\cdot1\right)$. Note that $a$ is the least
point in $C_{T}$, and thus we have $a\in D$. For every $2\leq i\leq n$
define the points $b_{i}=1_{i}\left(\left(n-1\right)\cdot\left(s-1\right)\right)$.
Then $\left(b_{i}\right)_{i}=1=a_{i}$ and $b_{i}\leq a$ for all
$i$, so (\ref{eq:10-1}) implies that $b_{i}\in d\left(D\right)$
for all $i$. For $2\leq i\le n$ define $c_{i}=0_{i}\left(\left(n-1\right)\cdot\left(s-1\right)\right)$,
and note that $c_{i}\in d\left(b_{i}\right)$ for all $i$. Hence
it follows that $c_{i}\in d\left(D\right)$ for all $2\leq i\leq n$. 

Let $x$ be a point with $x\in d\left(\left\{ 1,\dots,s-1\right\} ^{n}\right)$
and $x\neq w$. If $x_{i}=0$ for some $i\geq2$, then we have $x_{j}\leq\left(c_{i}\right)_{j}$
for every $j$. Since $d\left(A\right)$ is a down-set by Claim 4,
it follows that $x\in d\left(D\right)$. 

Now suppose that $x$ is a point with $x_{1}=0$ and $x\neq w$. Since
$x\neq w$ it follows that $R_{s-1}\left(x\right)$ is a proper subset
of $\left\{ 2,\dots,n\right\} $. Note that $R_{s-1}\left(c_{2}\right)=\left\{ 1,3,\dots,n\right\} $
is larger than any proper subset of $\left\{ 2,\dots,n\right\} $
under the binary order, and hence we must have $x\leq c_{2}$. Let
$k\in\left\{ 2,\dots,n\right\} \setminus R_{s-1}\left(x\right)$,
and consider $v$ obtained by taking $v_{j}=\left(c_{2}\right)_{j}$
for $j\neq k$ and $v_{k}=x_{k}$. Then we  have $x\leq v\leq c_{2}$,
and hence applying Claim 9 twice implies that $v\in d\left(D\right)$
and $x\in d\left(D\right)$. Thus we have $d\left(\left\{ 1,\dots,s-1\right\} ^{n}\right)\setminus\left\{ w\right\} \subseteq d\left(D\right)$,
which completes the proof. 
\end{proof}
Define the particular points $w=0\left(\left(n-1\right)\cdot\left(s-1\right)\right)$
and $a_{i}=i\left(\left(n-1\right)\cdot\left(s-1\right)\right)$ for
$1\leq i\leq s-1$. It is easy to verify that $w\in d\left(x\right)$
for  $x\in\left[s\right]^{n}\cap B_{0}$ if and only if $x=a_{i}$
for some $i$. In order to deal with the case $T=\left\{ 1\right\} $,
we prove a result that adding suitably many consecutive points to
an initial segment must increase the size of the $d$-shadow. This
is done in Claim 12, but first we need a preliminary result. For $x\in\left[k\right]^{n}$,
define the $x^{+}$ to be the successor of $x$, i.e.\ $x^{+}$ is
the least point under the $\leq$-order for which $x<x^{+}$. 
\begin{claim}
Let $x\in D$, and suppose that we also have $y=x^{+}\in D$. Then
the following claims are true. 
\end{claim}

\begin{enumerate}
\item There exists an unique $i$ satisfying $y_{i}>x_{i}$. 
\item $y_{j}<x_{j}$ implies that $y_{j}=1$. 
\item If $y_{j}\neq1$ for all $j$, it follows that $y_{j}>x_{i}$ for
all $j$, where $i$ is the unique index satisfying $y_{i}>x_{i}$. 
\end{enumerate}
\begin{proof}
Let $t$ be the largest index satisfying $R_{t}\left(x\right)\neq R_{t}\left(y\right)$. 
Since $x<y$ we must have $R_{t}\left(x\right)<_{bin}R_{t}\left(y\right)$.
Let $i=\max\left(R_{t}\left(x\right)\Delta R_{t}\left(y\right)\right)$,
and note that $i\in R_{t}\left(y\right)$ implies that $y_{i}>x_{i}$. 

Consider the point $z$ obtained by taking $z_{j}=\min\left(x_{j},y_{j}\right)$
for all $j\neq i$ and $z_{i}=y_{i}$. By the construction of $z$
we have $z_{j}\leq y_{j}$ for all $j$, and hence we have $z\leq y$.
For all $s>t$ we have $R_{s}\left(x\right)=R_{s}\left(y\right)$,
and hence by the construction of $z$ it also follows that $R_{s}\left(x\right)=R_{s}\left(z\right)$
for all $s>t$. Similarly, it is easy to see that $R_{t}\left(x\right)\cap R_{t}\left(y\right)\subseteq R_{t}\left(z\right)$.
Since we also have $i\in R_{t}\left(z\right)$, it follows that $\max\left(R_{t}\left(x\right)\Delta R_{t}\left(z\right)\right)=i\in R_{t}\left(z\right)$,
which implies that $x<z$. Combining these two observations we obtain
that $x<z\leq y$, and since $y=x^{+}$ it follows that $y=z$. Thus
for any $j\ne i$ we have $y_{j}=\min\left(x_{j},y_{j}\right)\leq x_{j}$,
which completes the proof of the first part. 

Let $X=\left\{ j:y_{j}<x_{j}\right\} $, and let $u$ be the sequence
obtained by taking $u_{j}=1$ for all $j\in X$ and $u_{j}=y_{j}$
for $j\not\in X$. Since $y\in D$, it follows that $u_{j}\leq y_{j}$
for all $j$, and hence we have $u\leq y$. Let $t$ and $i$ be defined
as in the previous part. Since $R_{s}\left(x\right)=R_{s}\left(y\right)$
for all $s>t$, it follows that $R_{s}\left(x\right)=R_{s}\left(u\right)$
for all $s>t$. Furthermore, it is easy to see that we have $\left\{ i\right\} \cup\left(R_{t}\left(x\right)\cap R_{t}\left(y\right)\right)\subseteq R_{t}\left(u\right)$.
Hence it follows that $\max\left(R_{t}\left(u\right)\Delta R_{t}\left(x\right)\right)=i$,
and therefore we have $x<u$. Combining these two observations we
obtain that $x<u\leq y$, and since $y=x^{+}$ it follows that $y=u$.
Hence $y_{j}<x_{j}$ implies that $y_{j}=1$, which proves the second
part. 

Suppose that we have $y_{j}\neq1$ for all $j$. Combining the first
and the second part, it follows that there exists unique $i$ so that
$y_{j}=x_{j}$ for all $j\neq i$ and $y_{i}>x_{i}$. Since $y$ is
the successor of $x$, it evidently follows that $y_{i}=x_{i}+1$.
Our aim is to prove that $x_{i}$ is strictly smaller than any other
$x_{j}$. First assume that there exists $k$ so that $x_{k}<x_{i}$,
and consider $v$ obtained by taking $v_{j}=x_{j}$ for $j\neq k$
and $v_{k}=x_{k}+1$. The we certainly have $x<v$. For convenience,
set $a=x_{i}$ and $b=x_{k}$. By the construction of $v$ it follows
that $R_{s}\left(x\right)=R_{s}\left(y\right)=R_{s}\left(v\right)$
for all $s\geq a+2$. However, since $b<a$ it follows that $R_{a+1}\left(y\right)=R_{a+1}\left(x\right)\cup\left\{ i\right\} $,
while $R_{a+1}\left(v\right)=R_{a+1}\left(x\right)$. Hence we also
have $v<y$ which contradicts the fact that $y$ is the successor
of $x$. 

Again let $a=x_{i}$, and note that if $a=1$ the claim follows evidently.
Now suppose that for all $j$ we have $x_{j}\geq a$ but $\left|R_{a}\left(x\right)\right|\geq2$.
Let $v$ be the point obtained by taking $v_{j}=x_{j}$ for all $j\not\in R_{a}\left(x\right)$,
$v_{i}=a+1$ and $v_{j}=1$ for all $j\in R_{a}\left(x\right)\setminus\left\{ i\right\} $.
Note that for all $s\ge a+2$ we have $R_{s}\left(x\right)=R_{s}\left(v\right)=R_{s}\left(y\right)$.
However, we also have $R_{a+1}\left(y\right)=R_{a+1}\left(v\right)=R_{a+1}\left(x\right)\cup\left\{ i\right\} $,
$R_{a}\left(y\right)=R_{a}\left(x\right)\setminus\left\{ i\right\} $
and $R_{a}\left(v\right)=\emptyset$ since $a>1$. Since $R_{a}\left(x\right)$
contains at least two elements, it follows that $x<v<y$, which contradicts
the fact that $y=x^{+}$. 

Hence for all $j\neq i$ we have $x_{j}>x_{i}$. Since $y_{j}\ne1$
for all $j$, the second part implies that $y_{j}\geq x_{j}$ for
all $j$. Hence for all $j\neq i$ we have $y_{j}\geq x_{j}>x_{i}$,
and by the first part we have $y_{i}>x_{i}.$ This completes the proof
of the third part. 
\end{proof}
\begin{claim}
Let $x_{1},\dots,x_{L-1}\in\left\{ 1,\dots,L-1\right\} ^{n}$ be consecutive
points under the $\leq$-order. Let $X$ and $Y$ be initial segments
of the $\leq$-order on $\left[k\right]^{n}$ defined by $X=\left\{ y:y\leq x_{L-1}\right\} $
and $Y=\left\{ y:y\leq x_{1}\right\} $. Then $\left|d\left(X\right)\right|=\left|d\left(Y\right)\right|$
if and only if $x_{i}=\left(i\right)\left(\left(n-1\right)\cdot\left(L-1\right)\right)$
for all $1\leq i\leq L-1$. 
\end{claim}

\begin{proof}
Let $x\in\left\{ 1,\dots,L-1\right\} ^{n}$, and consider the initial
segment $Z$ defined by $Z=\left\{ y:y<x\right\} $. We start by proving
that $\left|d\left(Z\cup\left\{ x\right\} \right)\right|=\left|d\left(Z\right)\right|$
if and only if $R_{1}\left(x\right)=\emptyset$. Suppose that there
exists $z\in d\left(Z\cup\left\{ x\right\} \right)\setminus d\left(Z\right)$,
and let $j$ be the unique index for which $z_{j}=0$. Note that such
$j$ is unique as $x\in B_{0}$. Let $a_{1},\dots,a_{L-1}$ be the
points obtained by taking $\left(a_{i}\right)_{k}=z_{k}$ for $k\neq j$,
and $\left(a_{i}\right)_{j}=i$. It is clear that $a_{1}<\dots<a_{L-1}$,
and we have $z\in d\left(u\right)$ for $u\in\left\{ 1,\dots,L-1\right\} ^{n}$
if and only if $u\in\left\{ a_{1},\dots,a_{L-1}\right\} $. In particular,
it follows that $z\in d\left(Z\cup\left\{ x\right\} \right)\setminus d\left(Z\right)$
if and only if $Z\cap\left\{ a_{1},\dots,a_{L-1}\right\} =\emptyset$
and $x\in\left\{ a_{1},\dots,a_{L-1}\right\} $. Since $Z$ is an
initial segment, we must have $x=a_{1}$. Thus $\left|d\left(Z\cup\left\{ x\right\} \right)\right|=\left|d\left(Z\right)\right|$
implies that $R_{1}\left(x\right)=\emptyset$, and if $x=a_{1}$ then
$z\in d\left(Z\cup\left\{ x\right\} \right)\setminus d\left(Z\right)$
which proves the converse as well. 

Let $X$ and $Y$ satisfy the conditions of the claim, and suppose
that we have $\left|d\left(X\right)\right|=\left|d\left(Y\right)\right|$.
Then the previous observation implies that we have $R_{1}\left(x_{i}\right)=\emptyset$
for all $2\leq i\leq L-1$. 

Let $m\left(i\right)$ denote the value of the smallest coordinate
of the point $x_{i}$. Since $R_{1}\left(x_{i}\right)=\emptyset$
for $2\leq i\leq L-1$, the third part of Claim 11 implies that $\left(x_{i+1}\right)_{j}>m\left(i\right)$
for all $j$. In particular, we have $m\left(i+1\right)>m\left(i\right)$,
and thus $m$ is strictly increasing. Since $R_{1}\left(x_{2}\right)=\emptyset$
it follows that $m\left(2\right)\geq2$, and hence we must have $m\left(i\right)\geq i$
for all $i$. In particular, it follows that $m\left(L-1\right)\geq L-1$,
and thus it follows that $x_{L-1}=\left(n\cdot\left(L-1\right)\right)$.
Since $x_{i}$'s are consecutive points, it is easy to verify that
$x_{i}=\left(i\right)\left(\left(n-1\right)\cdot\left(L-1\right)\right)$
for all $1\leq i\leq L-1$. 
\end{proof}
The following claim allows us to restrict to proving Theorem 1 for
compressed sets $A$ satisfying $B_{\geq1}\cup\left(\left[s\right]^{n}\cap B_{0}\right)\subseteq A\subseteq B_{\geq1}\cup\left(\left[s+1\right]^{n}\cap B_{0}\right)$
for some $s$. 
\begin{claim}
Let $A$ be a compressed set with $\left|A\right|\geq\left|B_{\geq1}\right|+2$
and let $s\geq2$ be chosen so that $\left|B_{\geq1}\cup\left(\left[s\right]^{n}\cap B_{0}\right)\right|<\left|A\right|\leq\left|B_{\geq1}\cup\left(\left[s+1\right]^{n}\cap B_{0}\right)\right|$.
Then there exists a compressed set $B$ with $\left|A\right|=\left|B\right|$
satisfying $B_{\geq1}\cup\left(\left[s\right]^{n}\cap B_{0}\right)\subseteq B\subseteq B_{\geq1}\cup\left(\left[s+1\right]^{n}\cap B_{0}\right)$
for which we have $\left|d\left(A\right)\right|\geq\left|d\left(B\right)\right|$. 
\end{claim}

\begin{proof}
Let $q=\max\left(m\left(x\right):x\in D\right)$, and note that the
condition on the size of $A$ implies that $q\geq s$. Let $T$ be
the largest element under the binary order so that the class $C_{T}$
has a non-empty intersection with $A$. 

First suppose that $T=\left\{ 1\right\} $, and let $C$ be the initial
segment of the $\leq$-order with $\left|C\right|=\left|A\right|$.
Define the particular points $w=0\left(\left(n-1\right)\cdot\left(q-1\right)\right)$
and $w_{i}=i\left(\left(n-1\right)\cdot\left(q-1\right)\right)$ for
$1\leq i\leq q-1$. Thus Claim 10 implies that $d\left(\left[s\right]^{n}\right)\setminus\left\{ w\right\} \subseteq d\left(A\right)$. 

We start by verifying that $C\cap C_{T}\subseteq A\cap C_{T}$. If
$q>s$, the inclusion is trivial as we have $C\cap C_{T}=\emptyset$.
If $q=s$, since $T$ is the largest element under the binary order
so that $C_{T}\cap A\neq\emptyset$, it follows that $A\setminus\left(B_{\geq1}\cup\left(\left[s\right]^{n}\cap B_{0}\right)\right)=A\cap C_{T}$,
and similarly $C\setminus\left(B_{\geq1}\cup\left(\left[s\right]^{n}\cap B_{0}\right)\right)=C\cap C_{T}$.
Since $C$ is an initial segment with $\left|C\right|=\left|A\right|$,
we must have $\left|A\cap C_{T}\right|\geq\left|C\cap C_{T}\right|$.
As $A\cap C_{T}$ is of the form $\left\{ y\in C_{T}:y\leq x\right\} $
for some $x\in C_{T}$, the inclusion follows. 

First consider the case when we also have $w\in d\left(A\right)$.
Then $d\left(\left[s\right]^{n}\right)\subseteq d\left(A\right)$,
and since $d\left(C\cap C_{T}\right)\subseteq d\left(A\right)$, it
follows that $d\left(C\right)\subseteq d\left(A\right)$. Thus we
may take $B=C$. 

Now suppose that $w\not\in d\left(A\right)$. We first deal with the
easy case $q>s$. Note that we have $d\left(C\right)\subseteq d\left(B_{\geq1}\cup\left(\left[s+1\right]^{n}\cap B_{0}\right)\right)$.
Since $q\geq s+1$, Claim 10 implies that $d\left(\left[s+1\right]^{n}\cap B_{0}\right)\setminus\left\{ w\right\} \subseteq d\left(A\right)$,
and hence it follows that $d\left(C\right)\setminus d\left(A\right)\subseteq\left\{ w\right\} $.
Since $q\geq s+1$, $A$ contains a point $u\in\left(\left[q\right]^{n}\setminus\left[s\right]^{n}\right)\cap B_{0}$.
As $n\geq2$ and $R_{0}\left(u\right)=\emptyset$, $d\left(u\right)$
contains a point whose coordinate is at least $s$. In particular,
it follows that $d\left(u\right)\not\subseteq d\left(C\right)$ and
hence we must have $\left|d\left(A\right)\setminus d\left(C\right)\right|\geq1$.
Combining this with $\left|d\left(C\right)\setminus d\left(A\right)\right|\leq1$,
it follows that $\left|d\left(A\right)\right|\geq\left|d\left(C\right)\right|$,
and hence we may take $B=C$. 

Now suppose that we have $w\not\in d\left(A\right)$ and $q=s$. By
using the observation $C\cap C_{T}\subseteq A\cap C_{T}$ and Claim
10, it follows that $d\left(C\right)\setminus d\left(A\right)=\left\{ w\right\} $.
Since $w\not\in d\left(A\right)$ and $q=s$, it follows that $\left\{ w_{1},\dots,w_{s-1}\right\} \subseteq C\setminus A$,
and in particular we have $\left|\left(C\setminus A\right)\cap\left(B_{\geq1}\cup\left(\left[s\right]^{n}\cap B_{0}\right)\right)\right|\geq s-1$.
Since $A\setminus\left(B_{\geq1}\cup\left(\left[s\right]^{n}\cap B_{0}\right)\right)=A\cap C_{T}$
and $C\setminus\left(B_{\geq1}\cup\left(\left[s\right]^{n}\cap B_{0}\right)\right)=C\cap C_{T}$,
it follows that $\left|A\cap C_{T}\right|\geq\left|C\cap C_{T}\right|+s-1$. 

Let $x$ be the largest point  of $A$ and $y$ be the largest point
of $C$ under the $\leq$-order. Define $X=\left\{ z:z\leq x\right\} $,
and note that $C=\left\{ z:z\leq y\right\} $ as $C$ is an initial
segment. Note that we have $x,y\in C_{T}$, and hence we have $\left|X\right|\geq\left|C\right|+\left(s-1\right)$.
Thus Claim 12 with $L=s$ implies that $\left|d\left(X\right)\right|>\left|d\left(C\right)\right|$,
unless $x=\left(2\right)\left(\left(n-1\right)\cdot\left(s\right)\right)$
and $y=\left(n\cdot s\right)$. However, these points are not in $C_{T}$
as $T=\left\{ 1\right\} $, and thus there exists $u\in d\left(X\right)\setminus d\left(C\right)$.
Thus we have $u\in d\left(\left\{ z:x<z\leq y\right\} \right)\subseteq d\left(C_{T}\cap A\right)\subseteq d\left(A\right)$,
and hence it follows that $\left|d\left(A\right)\setminus d\left(C\right)\right|\geq1$.
Since $\left|d\left(C\right)\setminus d\left(A\right)\right|=1$,
 it follows that $\left|d\left(A\right)\right|\geq\left|d\left(C\right)\right|$,
as required. 

Finally consider the case $T\neq\left\{ 1\right\} $. If $q>s$, Claim
10 implies that $d\left(\left[s+1\right]^{n}\cap B_{0}\right)\subseteq d\left(A\right)$,
and thus we can take $B=C$. If $q=s$, let $X=A\setminus\left(B_{\geq1}\cup\left(\left[s\right]^{n}\cap B_{0}\right)\right)$
and $k=\left|A\right|-\left|B_{\geq1}\cup\left(\left[s\right]^{n}\cap B_{0}\right)\right|$.
Let $Y$ be the set of $k$ least points in $X$ under the $\leq$-order,
and let $B=B_{\geq1}\cup\left(\left[s\right]^{n}\cap B_{0}\right)\cup Y$.
We certainly have $\left|A\right|=\left|B\right|$, and since $q=s$
it follows that $B\subseteq B_{\geq1}\cup\left(\left[s+1\right]^{n}\cap B_{0}\right)$.
Furthermore, as $A$ is compressed it follows that $B$ is compressed
as well. Since $T\neq\left\{ 1\right\} $, Claim 10 implies that $d\left(\left[s\right]^{n}\cap B_{0}\right)\subseteq d\left(A\right)$.
Since $Y\subseteq X$, it follows that $d\left(B\right)\subseteq d\left(A\right)$,
and thus $B$ satisfies the conditions, which completes the proof. 
\end{proof}
From now on we assume that $A$ is a compressed set for which there
exists $s\geq2$ so that $B_{\geq1}\cup\left(\left[s\right]^{n}\cap B_{0}\right)\subseteq A\subseteq B_{\geq1}\cup\left(\left[s+1\right]^{n}\cap B_{0}\right)$.
As before, we set $T$ to be the largest set under the binary order
for which $C_{T}\cap A\neq\emptyset$. Our next aim is to prove an
appropriate version of Claim 8. 
\begin{claim}
~
\begin{enumerate}
\item If $\left|T\right|\neq1$, then for all $S$ with $S<_{bin}T$
we have $C_{S}\subseteq A$. 
\item Let $T=\left\{ j\right\} $ and $U=\left\{ 1,\dots,j-1\right\} $
where $j\neq1$. Then for all $S$ with $S<_{bin}U$ we have $d\left(C_{S}\right)\subseteq d\left(A\right)$.
Define the particular point  $w_{j}=0_{j}\left(\left(\left(j-1\right)\cdot s\right)\left(\left(n-j\right)\cdot\left(s-1\right)\right)\right)$.
Then we also have $d\left(C_{U}\right)\setminus\left\{ w_{j}\right\} \subseteq d\left(A\right)$.
\end{enumerate}
\end{claim}

\begin{proof}
We start by proving the first part. Let $j=\max\left(T\right)$ and
set $U=\left\{ 1,\dots,j-1\right\} $. Since $\left|T\right|\neq1$
it follows that $U\cap T\neq\emptyset$, and hence (\ref{eq:10-1})
implies that $C_{U}\subseteq A$. Given any $S\in{\cal A}$ with $S<_{bin}T$,
note that we must either have $S\subseteq U$ or $j\in S$. If $S\subseteq U$,
we also have $S\leq_{bin}U$ and since $C_{U}\subseteq A$, (\ref{eq:10-1})
implies that $C_{S}\subseteq A$ holds as well. If $j\in S$, then
we certainly have $S\cap T\neq\emptyset$. Since $S<_{bin}T$, (\ref{eq:10-1})
implies that $C_{S}\subseteq A$. In particular, for all $S<_{bin}T$
we have $C_{S}\subseteq A$, which completes the proof of the first
part. \\

Now suppose that $T=\left\{ j\right\} $ for some $j$ with $j\not\in\left\{ 1,n\right\} $.
The assumption $j<n$ is used in the proof when we are considering
the $\left(j+1\right)^{th}$ coordinate of a sequence, and such applications
are not highlighted in the proof. Again, let $U=\left\{ 1,\dots,j-1\right\} $.
Define the particular points $a=s_{j}\left(\left(n-1\right)\cdot1\right)$
and $b=\left(\left(j-1\right)\cdot s\right)\left(\left(n-j+1\right)\cdot1\right)$.
Note that $a$ is the least point under the $\leq$-order in $C_{T}$
and $b$ is a point in $C_{U}$. Since $C_{T}\cap A\neq\emptyset$,
(\ref{eq:10-1}) implies that we also have $a\in A$. Since $U<_{bin}T$
and $a_{j+1}=b_{j+1}=1$, (\ref{eq:10-1}) implies that $b\in A$
and thus we have $C_{U}\cap A\neq\emptyset$. Let $S$ be a set satisfying
$S<_{bin}T$ and $S\neq U$. Then we must have $S\subset U$, and
since $C_{U}\cap A\neq\emptyset$, (\ref{eq:10-1}) implies that $C_{S}\subseteq A$.
Thus it only suffices to consider $d\left(C_{U}\right)$.

Let $x\in d\left(C_{U}\right)\setminus\left\{ w_{j}\right\} $ and
let $i$ be the unique index satisfying $x_{i}=0$. If $i\neq j$,
define $y$ by taking $y_{t}=x_{t}$ for all $t\neq i$ and $y_{i}=1$.
Since $j\neq i$, it follows that $y_{i}=a_{i}=1$. We also have $R_{s}\left(y\right)\subseteq U$,
and hence $y\leq a$. Since $a\in A$, (\ref{eq:10-1}) implies that
we must also have $y\in A$, and hence we have $x\in d\left(A\right)$. 

Now suppose that $i=j$. Since $x\in d\left(C_{U}\right)\setminus\left\{ w_{j}\right\} $
and $i=j$, it follows that $x_{t}=s$ for all $t\le j-1$, $x_{j}=0$,
$1\leq x_{t}\leq s-1$ for all $t\geq j+1$, and there exists $k\geq j+1$
satisfying $x_{k}\leq s-2$. Let $u$ be the point obtained by setting
$u_{t}=x_{t}$ for $t\neq j$ and $u_{j}=1$, and let $v$ be obtained
by taking $v_{t}=x_{t}$ for all $t\not\in\left\{ j,k\right\} $,
$v_{j}=s-1$ and $v_{k}=1$. It is easy to see that $x\in d\left(u\right)$.
Since $R_{s}\left(u\right)=R_{s}\left(v\right)$ and $R_{s-1}\left(v\right)=R_{s-1}\left(u\right)\cup\left\{ j\right\} $,
it follows that $u<v$. Since $k\neq i$, it follows that $v_{k}=a_{k}=1$,
and hence (\ref{eq:10-1}) implies that $v\in A$. Since $u_{1}=v_{1}$,
(\ref{eq:10-1}) implies that $u\in A$ and hence it follows that
$x\in d\left(A\right)$, as required. \\

Finally suppose that $T=\left\{ n\right\} $. Let $U=\left\{ 1,\dots,n-1\right\} $
and $V=\left\{ 2,\dots,n-1\right\} $. Note that for any $S<_{bin}T$
we have either $S=U$ or $S\leq_{bin}V$. Define the particular point
$a$ by $a=\left(\left(n-1\right)\cdot1\right)s$. Since $a$ is the
least point in $C_{T}$, it follows that $a\in A$. 

Given $S<_{bin}T$ and a point $x\in d\left(C_{S}\right)$, let
$i$ be the unique index for which $x_{i}=0$. Note that we may certainly
assume that we have $x\not\in d\left(C_{S'}\right)$ for any proper subset
$S'$ of $S$, with the convention $C_{\emptyset}=\left[s\right]^{n}\cap B_{0}$.
It is easy to see that this assumption is equivalent to assuming that $i\not\in S$. 

If $S=U$, the only possible point is $x=\left(\left(n-1\right)\cdot s\right)0$
which is precisely the point $w_{n}$. Now suppose that $S=V$, and
note that we must have $i\in\left\{ 1,n\right\} $. Let $u$ and $v$
be the points $u=\left(1\right)\left(\left(n-2\right)\cdot s\right)\left(s-1\right)$
and $v=\left(s-1\right)\left(\left(n-2\right)\cdot s\right)\left(1\right)$.
Since $a_{1}=u_{1}=1$ and $u<a$, (\ref{eq:10-1}) implies that $u\in A$.
Note that $R_{s}\left(u\right)=R_{s}\left(v\right)$ and $R_{s-1}\left(v\right)<_{bin}R_{s-1}\left(u\right)$,
and hence it follows that $v<u$. Since $n\geq3$, we have $u_{2}=v_{2}=s$,
and hence (\ref{eq:10-1}) implies that $v\in A$. 

For all $1\leq j\leq s-2$ and $1\leq k\leq s-2$ define the points
$w_{j,k}=\left(j\right)\left(\left(n-2\right)\cdot s\right)\left(k\right)$.
Since $R_{s-1}\left(w_{j,k}\right)=\emptyset$, it follows that $w_{j,k}<u$.
As $\left(w_{j,k}\right)_{2}=u_{2}=s$, (\ref{eq:10-1}) implies that
$w_{j,k}\in A$ for all $1\le j,k\leq s-2$.

Since $0\in\left\{ x_{1},x_{n}\right\} $, it follows that either
$x=\left(0\right)\left(\left(n-2\right)\cdot s\right)\left(c\right)$
or $x=\left(d\right)\left(\left(n-2\right)\cdot s\right)\left(0\right)$
for some $1\leq c,d\leq s-1$. If $c=s-1$ or $d=s-1$, we have $x\in d\left(u\right)$
or $x\in d\left(v\right)$ respectively. If $c\leq s-2$ or $d\leq s-2$,
it is easy to see that $x\in d\left(w_{1,c}\right)$ or $x\in d\left(w_{d,1}\right)$
respectively. In particular, it follows that $x\in d\left(A\right)$.
This completes the proof when $S=V$.

In order to deal with the sets $S$ with $S<_{bin}V$, we first assume
that $n\geq4$. As noticed in the previous case, we have $C_{V}\cap A\neq\emptyset$.
If $S\cap V\neq\emptyset$, (\ref{eq:10-1}) implies that we also
have $C_{S}\subseteq A$ and hence the conclusion certainly follows.
The only non-empty set $S$ satisfying $S<_{bin}V$ for which $S\cap V=\emptyset$
is $S=\left\{ 1\right\} $. However, as $n\geq4$ it follows that
$\left\{ 1\right\} <_{bin}\left\{ 1,2\right\} <_{bin}V$, and since
$\left\{ 1\right\} \cap\left\{ 1,2\right\} \neq\emptyset$ and $\left\{ 1,2\right\} \cap V\neq\emptyset$
we also have $C_{\left\{ 1\right\} }\subseteq D$. This completes
the proof when $n\geq4$. 

Now suppose that $n=3$. Then $V=\left\{ 2\right\} $, and thus the
only non-empty set $S$ satisfying $S<_{bin}V$ is $S=\left\{ 1\right\} $.
Let $x\in d\left(C_{\left\{ 1\right\} }\right)$ and let $i$ be the
unique index satisfying $x_{i}=0$. Recall that we may assume that
$i\in\left\{ 2,3\right\} $. Let $y$ be the point obtained by setting
$y_{j}=x_{j}$ for $j\neq i$ and $y_{i}=1$. Our aim is to show that
we have $y\in D$. 

If $i=2$ and $y_{3}=s-1$, we have $y_{3}=u_{3}$. Since $y\leq u$,
(\ref{eq:10-1}) implies that $y\in D$. If $i=2$ and $y_{3}=c$
for some $1\leq c\leq s-2$, then $y_{3}=c=\left(w_{1,c}\right)_{3}$,
and since $y\leq w_{1,c}$ it follows that $y\in D$. Finally, if
$i=3$, we have $y_{3}=v_{3}=1$, and since $y<v$ it follows that
$y\in D$. Thus in all three cases we have $y\in D$, which completes
the proof. 
\end{proof}
Now we have all the necessary tools to finish the proof of Theorem
1. Let $A$ be a set satisfying $B_{\geq1}\cup\left(\left[s\right]^{n}\cap B_{0}\right)\subseteq A\subseteq B_{\geq1}\cup\left(\left[s+1\right]^{n}\cap B_{0}\right)$,
and let $C$ be the initial segment of $\leq$ with $\left|A\right|=\left|C\right|$.
As $A$ and $C$ are of the same size, we must also have $B_{\geq1}\cup\left(\left[s\right]^{n}\cap B_{0}\right)\subseteq C\subseteq B_{\geq1}\cup\left(\left[s+1\right]^{n}\cap B_{0}\right)$.
Let $T_{1}$ be the largest class under the binary order satisfying
$C_{T'}\cap C\neq\emptyset$, and since $C$ is an initial segment
we certainly have $T_{1}\leq_{bin}T$. 

First suppose that $T_{1}<_{bin}T$, and set 
\[
X=B_{\geq1}\cup\left(\left[s\right]^{n}\cap B_{0}\right)\cup\bigcup_{S\leq_{bin}T_{1}}C_{S}.
\]
Note that we cannot have $T=\left\{ 1\right\} $ as $\left\{ 1\right\} $
is the least point under the binary order in ${\cal A}$. 

If $\left|T\right|\geq2$, for any $S\in{\cal A}$ with $S\leq_{bin}T_{1}$
we have $S<_{bin}T$, and hence Claim 14 implies that $C_{S}\subseteq A$.
Hence it follows that $X\subseteq A$, and since $C\subseteq X$ we
must have $C\subseteq A$. Since $A$ and $C$ have the same size,
it follows that $A=C$. However, this contradicts the assumption $T_{1}<_{bin}T$,
so this case cannot occur when $\left|T\right|\geq2$. 

If $T=\left\{ j\right\} $ for some $j\neq1$, define the particular
point $w_{j}=0_{j}\left(\left(\left(j-1\right)\cdot s\right)\left(\left(n-j\right)\cdot\left(s-1\right)\right)\right)$.
Then Claim 14 implies that for all $S<_{bin}\left\{ 1,\dots,j-1\right\} $
we have $d\left(C_{S}\right)\subseteq d\left(A\right)$, and for $U=\left\{ 1,\dots,j-1\right\} $
we have $d\left(C_{U}\right)\setminus\left\{ w_{j}\right\} \subseteq d\left(A\right)$.
Since $T_{1}<_{bin}T$, it follows that $T_{1}\leq_{bin}U$, and hence
$d\left(X\right)\setminus\left\{ w_{j}\right\} \subseteq d\left(A\right)$.
In particular, we have $\left|d\left(X\right)\setminus d\left(A\right)\right|\leq1$. 

Pick any point $v\in C_{T}$, and consider $u\in d\left(v\right)$
obtained by flipping the first coordinate to $0$. Since $j\neq1$,
it follows that $u_{j}=s$. In particular, it follows that $u\not\in d\left(X\right)$
as $T_{1}\leq_{bin}U$. Hence $u\in d\left(A\right)\setminus d\left(X\right)$,
and thus it follows that $\left|d\left(A\right)\setminus d\left(X\right)\right|\geq1$.
Hence we have $\left|d\left(A\right)\right|\geq\left|d\left(X\right)\right|$,
and since $C\subseteq X$ this completes the proof in this case. 

Hence we are only left with the case when $T_{1}=T$. It turns out
to be convenient to split into two subcases based on the size of $T$.
When $T=\left\{ 1\right\} $, there are no non-empty sets $S$ satisfying
$S<_{bin}T$. Hence we can deal with the cases $T=\left\{ 1\right\} $
and $\left|T\right|\geq2$ at the same time. 

\subsubsection*{Case 2.1. $\left|T\right|\protect\geq2$ or $T=\left\{ 1\right\}.$}

Let $C$ be an initial segment with $\left|C\right|=\left|A\right|$,
and let $s$, $T$ and $T_{1}$ be chosen as before. Recall that we
have $T_{1}=T$. We prove that under these assumptions it follows that
$A=C$. 

Let $X$ be defined by $X=B_{\geq1}\cup\left(\left[s\right]^{n}\cap B_{0}\right)\cup\left(\bigcup_{S<_{bin}T}C_{S}\right)$.
Claim 14 implies that for all $S<_{bin}T$ we have $C_{S}\subseteq A$,
and hence it follows that $X\subseteq A$. On the other hand, by the
choice of $T$ we have $A\setminus X=A\cap C_{T}$ and $C\setminus X=C\cap C_{T}$.
Since $\left|A\right|=\left|C\right|$ and both of these sets contain
$X$ as a subset, it follows that $\left|A\cap C_{T}\right|=\left|C\cap C_{T}\right|$.
Since both sets $A\cap C_{T}$ and $C\cap C_{T}$ are of the form
$\left\{ y\in C_{T}:y\le x\right\} $ for some $x\in C_{T}$, it follows
that $A\cap C_{T}=C\cap C_{T}$. Hence we must have $A=C$, as required. 

\subsubsection*{Case 2.2. $T=\left\{ j\right\} $ where $2\protect\leq j\protect\leq n.$}

As before, let $C$ be the initial segment with $\left|C\right|=\left|A\right|$,
let $s$, $T$ and $T_{1}$ be chosen as before, and recall that we
have $T_{1}=T=\left\{ j\right\} $ for some $2\leq j\leq n$. Define
the particular points $w=0_{j}\left(\left(\left(j-1\right)\cdot s\right)\left(\left(n-j\right)\cdot\left(s-1\right)\right)\right)$
and $w_{i}=i_{j}\left(\left(\left(j-1\right)\cdot s\right)\left(\left(n-j\right)\cdot\left(s-1\right)\right)\right)$
for $1\le i\leq s-1$. Note that if $w\in d\left(x\right)$ for some
$x\in C_{S}$ with $S\leq_{bin}T$, then we must have $x=w_{i}$ for
some $i$. As before, let $X=B_{\geq1}\cup\left(\left[s\right]^{n}\cap B_{0}\right)\cup\bigcup_{S<_{bin}T}C_{S}$. 

First consider the easy case when we have $w\in d\left(A\right)$.
Then Claim 10 implies that $d\left(C_{S}\right)\subseteq d\left(A\right)$
for all $S$ with $S<_{bin}T$. Since $X\subseteq C$ and we have
$C\setminus X\subseteq C_{T}$ and $A\setminus X\subseteq C_{T}$,
it follows that $\left|A\cap C_{T}\right|\geq\left|C\cap C_{T}\right|$.
By using the same argument as in Case 2.1, it follows that $C\cap C_{T}\subseteq A\cap C_{T}$
and thus it follows that $d\left(C\right)\subseteq d\left(A\right)$,
as required. 

Now suppose that $w\not\in d\left(A\right)$, and thus it follows
that $\left\{ w_{1},\dots,w_{s-1}\right\} \cap A\neq\emptyset$. Since we have 
$\left\{ w_{1},\dots,w_{s-1}\right\} \subseteq X\subseteq C$, it follows that 
$\left|C\setminus X\right|\geq\left|A\setminus X\right|+s-1$,
and hence we have $\left|A\cap C_{T}\right|\geq\left|C\cap C_{T}\right|+s-1$.
Let $u$ and $v$ be chosen so that $A\cap C_{T}=\left\{ z\in C_{T}:z\leq u\right\} $
and $C\cap C_{T}=\left\{ z\in C_{T}:z\leq v\right\} $. Let $U$ be
the initial segment $U=\left\{ z:z\leq u\right\} $, and note that
$C=\left\{ z:z\leq v\right\} $. Since $\left|C\setminus X\right|\geq\left|A\setminus X\right|+s-1$,
it follows that $\left|U\right|\geq\left|C\right|+s-1$. 

Since $n\geq3$, it follows that $\left(1\right)\left(\left(n-1\right)\cdot s\right)\not\in C_{T}$,
and hence must have $v\neq\left(1\right)\left(\left(n-1\right)\cdot s\right)$.
Since $\left|U\right|\geq\left|C\right|+s-1$, Claim 12 with $L=s+1$
implies that there exists $a\in d\left(U\right)\setminus d\left(C\right)$.
Hence we must have $a\in d\left(A\cap C_{T}\right)$, and thus it
follows that $\left|d\left(A\right)\setminus d\left(C\right)\right|\geq1$.
Together with $\left|d\left(C\right)\setminus d\left(A\right)\right|=1$,
it follows that $\left|d\left(A\right)\right|\geq\left|d\left(C\right)\right|$. 
This completes the proof of Theorem 1. $\hfill\square$

\section{Minimal $d$-shadow for a given rank}

Recall that we defined $\left[k\right]_{r}^{n}=\left\{ x\in\left[k\right]^{n}:w\left(x\right)=r\right\} $
to be the set of those sequences with exactly $r$ non-zero coordinates,
and consider the restriction of the $\leq$-order on $\left[k\right]_{r}^{n}$.
Since $\left|R_{0}\left(x\right)\right|=n-r$ is constant for all
$x\in\left[k\right]_{r}^{n}$, it follows that for distinct $x$ and
$y$ we have $x\le y$ if and only if $\max\left(R_{j}\left(x\right)\Delta R_{j}\left(y\right)\right)\in R_{j}\left(y\right)$,
where $j$ is the largest index for which $R_{j}\left(x\right)\neq R_{j}\left(y\right)$. 

Note that for all $m\leq k$, $\left[m\right]_{r}^{n}$ is an initial
segment of the $\leq$-order in $\left[k\right]_{r}^{n}$, and the
restrictions of the $\leq$-order on $\left[m\right]_{r}^{n}$ and
on $\left[k\right]_{r}^{n}$ coincide on $\left[m\right]_{r}^{n}$.
Thus these orders extend naturally to an order on $\mathbb{N}_{r}^{n}=\left\{ x=x_{1}\dots x_{n}:x_{i}\in\mathbb{N},w\left(x\right)=r\right\} $,
which we will also denote by $\leq$. 

The notion of $d$-shadow is still sensible for subsets of $\mathbb{N}_{r}^{n}$
as well. If $A\subseteq\mathbb{N}_{r}^{n}$, we certainly have $d\left(A\right)\subseteq\mathbb{N}_{r-1}^{n}$.
We  now use Theorem 1 to deduce that among the subsets of $\mathbb{N}_{r}^{n}$
of a given finite size, the initial segment of the $\leq$-order minimises
the size of the $d$-shadow. 
\addtocounter{thm}{-13}
\begin{thm}
Let $A$ be a finite subset of $\mathbb{N}_{r}^{n}$ and let $C$
be the initial segment of the $\leq$-order on $\mathbb{N}_{r}^{n}$
with $\left|A\right|=\left|C\right|$. Then we have $\left|d\left(A\right)\right|\geq\left|d\left(C\right)\right|$. 
\end{thm}
\begin{proof}
Let $k$ be the size of $A$. Since $A$ contains points of length
$n$ and exactly $r$ coordinates that equal zero, by reordering coordinates
if necessary we may assume that $A\subseteq\left[nk\right]_{r}^{n}$.
Let $B=\left[nk\right]_{\leq r-1}^{n}\cup A$, and recall that we
have $\left[nk\right]_{\le r-1}^{n}=\left\{ x\in\left[nk\right]^{n}:\left|R_{0}\left(x\right)\right|\geq n-r+1\right\} $.
Let $X$ be the initial segment on $\left[nk\right]^{n}$ with $\left|X\right|=\left|B\right|$.
Then $X=\left[nk\right]_{\leq r-1}^{n}\cup C$, where $C$ is the
initial segment of the $\leq$-order on $\left[nk\right]_{r}^{n}$ with $\left|C\right|=\left|A\right|$.
Thus $C$ is also the initial segment of $\leq$ on $\mathbb{N}_{r}^{n}$
with $\left|C\right|=\left|A\right|$. 

Note that we have 
\[
\left|d\left(B\right)\right|=\left|\left[nk\right]_{\leq r-1}^{n}\right|+\left|d\left(A\right)\right|
\]
and 
\[
\left|d\left(X\right)\right|=\left|\left[nk\right]_{\leq r-1}^{n}\right|+\left|d\left(C\right)\right|,
\]
as 
\[
d\left(A\right)\cap\left[nk\right]_{\leq r-1}^{n}=d\left(C\right)\cap\left[nk\right]_{\leq r-1}^{n}=\emptyset.
\]

Since $\left|B\right|=\left|X\right|$ and $X$ is an initial segment
of the $\leq$-order on $\left[nk\right]^{n}$, Theorem 1 implies
that $\left|d\left(B\right)\right|\geq\left|d\left(X\right)\right|$.
Thus we have $\left|d\left(A\right)\right|\geq\left|d\left(C\right)\right|$,
which completes the proof. 
\end{proof}

\end{document}